\documentclass[a4paper,10pt]{article}

\usepackage{geometry}
\usepackage{amsthm,amsmath,amssymb}
\usepackage[OT1]{fontenc}
\usepackage{graphicx}
\usepackage[utf8]{inputenc}
\usepackage{enumitem}

\usepackage{tabularx,booktabs}

\usepackage{authblk}

\newtheorem{theorem}{Theorem}
\newtheorem{lemma}{Lemma}[section]

\newtheorem{observation}{Observation}
\newtheorem{corollary}[theorem]{Corollary}
\newtheorem{question}{Question}

\newtheorem{conjecture}{Conjecture}

\graphicspath{{./fig/}}

\newlist{proofEnum}{enumerate}{1}
\setlist[proofEnum]{wide, labelwidth=!, labelindent=0pt, label=(\alph{proofEnumi})}

\let\leq\leqslant
\let\geq\geqslant

\newcommand{\cH}{\mathcal{H}}
\newcommand{\cE}{\mathcal{E}}
\newcommand{\cF}{\mathcal{F}}

\newcommand{\cP}{\mathcal{P}}

\newcommand{\bN}{\mathbb{N}}

\newcommand{\set}[1]{{\left\{#1\right\}}}
\newcommand{\floor}[1]{{\left\lfloor #1 \right\rfloor}}

\newcommand{\rs}{R} 
\newcommand{\rn}{r} 
\newcommand{\oto}{\to}
\newcommand{\ors}{R_<}
\newcommand{\orn}{r_<}
\newcommand{\ord}{r_<^m}

\newcommand{\con}{\circ}
\newcommand{\hang}[3]{#1\oplus_{#2}#3}
\newcommand{\uo}[1]{\widetilde{#1}}
\newcommand{\girth}{{\rm girth}} 

\begin{document}
\title{Minimal Ordered Ramsey Graphs}

\author{Jonathan Rollin}
\affil{\small Department of Mathematics, Karlsruhe Institute of Technology} 
\date{}

\maketitle


\begin{abstract}
 An \emph{ordered graph} is a graph equipped with a linear ordering of its vertex set.
 A pair of ordered graphs is \emph{Ramsey finite} if it has only finitely many minimal ordered Ramsey graphs and \emph{Ramsey infinite} otherwise.
 Here an ordered graph $F$ is an \emph{ordered Ramsey graph} of a pair $(H,H')$ of ordered graphs if for any coloring of the edges of $F$ in colors red and blue there is either a copy of $H$ with all edges colored red or a copy of $H'$ with all edges colored blue.
 Such an ordered Ramsey graph is \emph{minimal} if neither of its proper subgraphs is an ordered Ramsey graph of $(H,H')$.
 If $H=H'$ then $H$ itself is called Ramsey finite.
 
 We show that a connected ordered graph is Ramsey finite if and only if it is a star with center being the first or the last vertex in the linear order. 
 In general we prove that each Ramsey finite (not necessarily connected) ordered graph $H$ has a pseudoforest as a Ramsey graph and therefore is a star forest with strong restrictions on the positions of the centers of the stars.
 In the asymmetric case we show that $(H,H')$ is Ramsey finite whenever $H$ is a so-called monotone matching.
 Among several further results we show that there are Ramsey finite pairs of ordered stars and ordered caterpillars of arbitrary size and diameter.
 This is in contrast to the unordered setting where for any Ramsey finite pair $(H,H')$ of forests either one of $H$ or $H'$ is a matching or both are star forests (with additional constraints).
 
 
 Several of our results give a relation between Ramsey finiteness and the existence of sparse ordered Ramsey graphs.
 Motivated by these relations we characterize all pairs of ordered graphs that have a forest as an ordered Ramsey graph and all pairs of connected ordered graphs that have a pseudoforest as a Ramsey graph.

 Our results show similarities between the ordered and the unordered setting for graphs containing cycles and significant differences for forests.
\end{abstract}

\section{Introduction}\label{sec:intro}

Graph Ramsey theory is concerned with the phenomenon that for any given graph $H$ there are graphs $F$, called the \emph{Ramsey graphs of $H$}, such that for any $2$-coloring of the edges of $F$ there is a copy of $H$ in $F$ with all its edges of the same color.
For most graphs $H$ it is a challenging problem to determine all its Ramsey graphs exactly which is solved only for few classes of graphs like small matchings or stars~\cite{AllRamseyMinP3,BEL76}.
Therefore, particular properties of the set $\rs(H)$ of all Ramsey graphs of $H$ and its members are studied.
This line of research was initiated by fundamental work of Ne{\v{s}}et{\v{r}}il and R\"odl~\cite{CliqueNoRamseyGeneral} and Burr, Erd\H{o}s, and Lov\'asz~\cite{BEL76}.
One of the most famous questions asks for the smallest number of vertices of graphs in $\rs(H)$ called the Ramsey number and denoted $\rn(H)$.
Determining the Ramsey number of complete graphs is a challenging problem on its own and no exact formula is known yet.

In this paper we study some structural questions from the Ramsey theory for ordered graphs.
Here an \emph{ordered graph} is a graph equipped with a linear ordering of its vertex set.
An \emph{(ordered) subgraph} of an ordered graph $G$ is a subgraph of the underlying graph of $G$ that inherits the ordering of vertices from $G$.
Analogously to the unordered setting we have the following definitions.
An ordered graph $F$ is an \emph{ordered Ramsey graph} of some pair $(H,H')$ of ordered graphs if for any coloring of the edges of $F$ in colors red and blue there is either a copy of $H$ with all its edges colored red or a copy of $H'$ with all its edges colored blue.
In this case we write $F\oto (H,H')$ to indicate this fact and let $\ors(H,H')=\{F\mid F\oto(H,H')\}$.
If $H=H'$, then we write $F\oto H$ and $\ors(H)=\ors(H,H)$.
A fundamental relation between ordered and unordered Ramsey graphs is given in the following observation.
\begin{observation}\label{obs:basicRelOrderedRamsey}
 Let $F$ and $H$ be an ordered graphs with underlying (unordered) graphs $\uo{F}$ and $\uo{H}$.
 If $F$ is an ordered Ramsey graph of $H$, then $\uo{F}$ is a Ramsey graph of $\uo{H}$.
\end{observation}
 If $H$ is a complete graph, then also the reverse statement holds, otherwise it may fail.
 Figure~\ref{fig:exampleOrderedRamsey1} shows an example of an ordered graph $H$ and a Ramsey graph $F'$ of $\uo{H}$ which does not form an ordered Ramsey graph of $H$ in any ordering.
\begin{figure}
 \centering
 \includegraphics{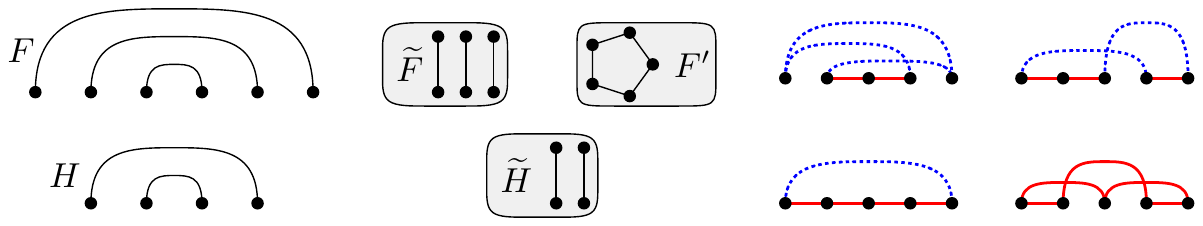}
 \caption[Ordered Ramsey graphs.]{An ordered matching $H$ and an ordered Ramsey graph $F$ of $H$ (left). An (unordered) $5$-cycle $F'$ is a Ramsey graph of the underlying graph $\uo{H}$ of $H$, but no ordering of the vertices of $F'$ yields an ordered Ramsey graph of $H$ (right). Here we show just a few possible orderings of $F'$. In~\cite{RollinDiss}  we show that any ordered Ramsey graph of $H$ contains a copy of $F$ (which is not possible for orderings of $F'$).}
 \label{fig:exampleOrderedRamsey1}
\end{figure}

We study the fundamental question whether the set $\ors(H,H')$ has a finite number of minimal elements, where $F\in\ors(H,H')$ is \emph{minimal} if $F'\not\in\ors(H,H')$ for each proper subgraph $F'$ of $F$.
In this case we call the pair $(H,H')$ \emph{Ramsey finite}.
The corresponding question in the unordered setting is studied intensively (see Theorems~\ref{thm:UnorderedCycleInf}, \ref{thm:unorderedForestCycle}, \ref{thm:Faudree} below) but a full answer is not known.
Our results indicate that any Ramsey finite pair of ordered graphs necessarily has sparse ordered Ramsey graphs.
Motivated by this relations we first give several results on sparse ordered Ramsey graphs, which might be of independent interest.

\paragraph{Outline.}
Next we present some basic definitions which are used throughout the paper.
In particular some frequently used notions for ordered graphs are introduced.
Then we present our results on sparse ordered Ramsey graphs, followed by our results on minimal ordered Ramsey graphs.
A brief summary of previous work on Ramsey theory for ordered graphs is given at the end of Section~\ref{sec:intro}.
In Section~\ref{sec:proofs} we prove our results and Section~\ref{sec:conclusion} contains concluding remarks and open questions.

\paragraph{Preliminary Remarks and Definitions.}
Before stating our results we need to introduce some notions.
For a positive integer $n$ we shall  write $[n]=\{1,\ldots,n\}$.
For a given (ordered) graph $G$ we refer to its vertex set by $V(G)$ and to its edge set by $E(G)$.
We consider the vertices of an ordered graph laid out along a horizontal line from left to right such that a vertex $u$ is \emph{to the left} of a vertex $v$ if $u<v$, and to the right if $v<u$.
For two sets $U$, $U'\subseteq V(G)$ we write $U\preceq U'$ ($U\prec U'$) if $u\leq u'$ ($u<u'$) for all $u\in U$ and $u'\in U'$.
For two subgraphs of $G'$, $G''$ of $G$ we write $G'\preceq G''$ ($G'\prec G''$) if $V(G')\preceq V(G'')$ ($V(G')\prec V(G'')$).
An \emph{interval} of an ordered graph $G$ is a set $I$ of consecutive vertices of $G$, i.e., for any $u$, $v\in I$, $z\in V(G)$, with $u\leq z\leq v$, we have $z\in I$.

A complete graph on $n$ vertices is denoted $K_n$.
A path on $n$ vertices is denoted $P_n$ and an ordered path $P=u_1\cdots u_n$ is a \emph{monotone path} if $u_1<\cdots<u_n$.
A \emph{partial matching} is a graph without any copy of $P_3$.
A \emph{right star} is an ordered star with all its leaves to the right of its center and a \emph{left star} is defined accordingly.
A right star with $k$ leaves is denoted $\vec{S}_k$.
Right and left stars with exactly two edges are also called \emph{bend}.
An ordered graph $G$ is a \emph{right caterpillar} if it is connected and for some $i\geq 1$ there are vertices $u_{i}>\cdots>u_0$ in $G$ such that each edge in $G$ is of the form $u_jv$ with $j\in[i]$ and $u_j\geq v\geq u_{j-1}$.
For $j\in[n]$ the \emph{$j^\text{th}$ segment} of a right caterpillar $G$ is the subgraph $S_j$ of $G$ induced by $\{v\in V(G)\mid u_j\geq v\geq u_{j-1}\}$.
The \emph{defining sequence} of $G$ is $|E(S_1)|,\ldots,|E(S_i)|$.
Left caterpillars are defined accordingly.
See Figure~\ref{fig:namedOrderedGraphs} for examples.
\begin{figure}
 \centering
 \includegraphics{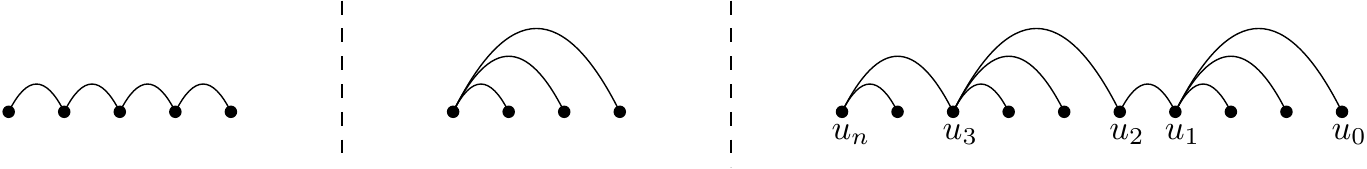}
 \caption{A monotone path (left), a right star (middle), and a right caterpillar with four segments (right).}
 \label{fig:namedOrderedGraphs}
\end{figure}

\paragraph{Sparse Ordered Ramsey Graphs.}

First, we state some known results on densities of (unordered) Ramsey graphs.
Given a graph $G$ let $m(G)=\max\{|E(G')|/|V(G')|\mid G'\subseteq G\}$ denote its \emph{density} and let $m_2(G) =\max\{(\lvert E(G')\rvert-1)/(\lvert V(G')\rvert-2)\mid G'\subseteq G, \lvert V(G')\rvert\geq 3, \lvert E(G')\rvert\geq 1\}$ denote its \emph{$2$-density}, provided that $G$ has at least two edges.
For a pair $(H,H')$ of graphs let $r^m(H,H')=\inf\{m(F)\mid F\in\rs(H,H')\}$ denote its \emph{Ramsey density}.
The Ramsey density is studied by R\"odl and Ruci\'nski~\cite{RR93}, Kurek and Ruci\'nski~\cite{KR05}, and M\"utze and Peter~\cite{MP13}.
The exact value is known only for few classes of graphs, including stars and complete graphs.
The following result is due to R\"odl and Ruci\'nski~\cite{RR93} (see also Nenadov and Steger~\cite{NS16}).
\begin{theorem}[\cite{RR93}]\label{thm:avgDegRamsey}
 Let $H$ be a graph with $m_2(H)>1$. Then $m(F)>m_2(H)$ for each $F\in\rs(H)$.
\end{theorem}
This leaves to consider graphs with $m_2(H)\leq 1$.
Observe that $m_2(H)\leq 1$ if and only if $H$ is a forest.
Recall that a partial matching is a graph without any copy of $P_3$.
More precisely we have $m_2(H)= 1/2$ if and only if $H$ is a partial matching and $m_2(H)= 1$ if and only if $H$ is a forest which is not a partial matching.
Further observe that $m(F)<1$ if and only if $F$ is a forest and $m(F)=1$ if and only if each component of $F$ contains at most one cycle and $F$ is not a forest.
Graphs of density at most $1$ are called \emph{pseudoforests}.
A \emph{proper pseudoforest} is a graph of density exactly $1$, that is, a pseudoforest that contains at least one cycle.
It seems well-known which pairs of graphs have Ramsey density at most $1$.
We give a proof of the following result in Section~\ref{sec:proofRamseyForest} for completeness.
\begin{lemma}\label{lem:unorderedPseudo}
 For each pair $(H,H')$ of graphs with at least one edge each we have
 \begin{enumerate}
  \item $r^m(H,H') < 1$ if and only if $(H,H')$ is a pair of a forest and a star forest,\label{enum:unorderedRamseyForest}
  
  \item $r^m(H,H') = 1$ and there is a pseudoforest in $\rs(H,H')$ if and only if $(H,H')$ is a pair of a partial matching and a proper pseudoforest or both $H$ and $H'$ are forests of stars and copies of $P_4$, both with at least one copy of $P_4$.
 \end{enumerate}
\end{lemma}
With Theorem~\ref{thm:avgDegRamsey} we obtain the following corollary, implicitly in~\cite{FK00,RR93}.
Note that for any matching $H$ there is a matching $F\in\rs(H)$, that is, $m(F)=m_2(H)=1/2$.
\begin{corollary}[\cite{FK00,RR93}]\label{cor:unorderedDens2Dens}
 Let $H$ be a graph. Then there is a graph $F\in\rs(H)$ with $m(F)\leq m_2(H)$ if and only if each component of $H$ is either a star or a copy of $P_4$.
\end{corollary}
Here we are interested in the corresponding parameter $\ord(H,H')=\inf\{m(F)\mid F\in\ors(H,H')\}$ for a pair $(H,H')$ of ordered graphs.
We completely determine the pairs of ordered graphs that have forests as ordered Ramsey graphs in the following theorem.
\begin{theorem}\label{thm:OrderedRamseyForest}
 Let $H$ and $H'$ be ordered graphs with at least one edge each.
 Then $\ord(H,H')<1$ if and only if $H$ and $H'$ are forests and one of the following statements holds.
 \begin{enumerate} 
  \item $H$ or $H'$ is a partial matching.\label{enum:matching}
  \item For one of $H$ or $H'$ each component is a right star and for the other each vertex has at most one neighbor to the left.\label{enum:rightForests}
  \item For one of $H$ or $H'$ each component is a left star and for the other each vertex has at most one neighbor to the right.\label{enum:leftForests}
  \item For one of $H$ or $H'$ each component is a left or a right star and for the other each component is a monotone path.\label{enum:RightLeftStars}
 \end{enumerate}
 Moreover $\ors(H,H')$ contains a partial matching if and only if both $H$ and $H'$ are partial matchings.
\end{theorem}
Further we characterize all pairs of connected ordered graphs that have pseudoforests as Ramsey graphs.
\begin{theorem}\label{thm:OrderedRamseyPseudoforest}
 Let $H$ and $H'$ be connected ordered graphs with at least one edge each.
 Then $\ord(H,H')=1$ and there is an ordered pseudoforest in $\ors(H,H')$ if and only if $(H,H')$ is a pair of $K_2$ and a connected ordered proper pseudoforest or both $H$ and $H'$ form a monotone $P_3$.
\end{theorem}
\begin{corollary}\label{cor:orderedDens2Dens}
 Let $H$ be a connected ordered graph. Then there is an ordered graph $F\in\ors(H)$ with $m(F)\leq m_2(H)$ if and only if $H$ is a left star, a right star, or a monotone $P_3$.
\end{corollary}

\paragraph{Ramsey (In)Finiteness.}

Recall that a graph $F\in\rs(H,H')$ is \emph{minimal} if $F'\not\in\rs(H,H')$ for each proper subgraph $F'$ of $F$.
A pair $(H,H')$ of graphs is \emph{Ramsey finite} if there are only finitely many minimal graphs in $\rs(H,H')$, and \emph{Ramsey infinite} otherwise.
%
%
%
The following theorems summarize some of the known main results in the unordered setting.
\begin{theorem}[\cite{RR93,RR95}, see~\cite{Lu94}]\label{thm:UnorderedCycleInf}
 If a graph $H$ contains a cycle, then $H$ is Ramsey infinite.
\end{theorem}
\begin{theorem}[\cite{B78},\cite{Lu94}]\label{thm:unorderedForestCycle}
Let $H$ be a forest without isolated vertices and let $H'$ be graph with at least one cycle.
\begin{enumerate}
 \item If $H$  is a matching, then $(H,H')$ is Ramsey finite.\label{enum:unordredMatchingFinite}
 
 \item If $H$  is not a matching, then $(H,H')$ is Ramsey infinite.\label{enum:unordredNotMatchingInfinite}
\end{enumerate}
\end{theorem}
\begin{theorem}[\cite{Fa91}]\label{thm:Faudree}
 Let $H$ and $H'$ be forests without isolated vertices. Then there is a constant $n_0$ such that $(H,H')$ is Ramsey finite if and only if one of the following statements holds.
 \begin{enumerate} 
  \item At least one of $H$ or $H'$ is a matching.
  \item Both $H$ and $H'$ are vertex disjoint unions of a star with an odd number of edges and a matching.
  \item One of $H$ or $H'$ is a vertex disjoint union of a matching and at least two stars with $m_1$ respectively $m_2$ edges, while the other is a vertex disjoint union of a matching on $n$ edges and a star with $n_1$ edges. Moreover $m_1$, $n_1$ are odd, $m_1\geq n_1+m_2-1$, and $n\geq n_0$.
 \end{enumerate}
\end{theorem}
The asymmetric case for a pair $(H,H')$ of graphs containing a cycle is not completely resolved.
Ne\v{s}et\v{r}il and R\"odl~\cite{NR78} prove that $(H,H')$ is Ramsey infinite if both $H$ and $H'$ are $3$-connected or both are of chromatic number at least $3$ while results in~\cite{BFS85} show that it is sufficient to consider pairs of $2$-connected graphs.
Bollob\'as~\textit{et al.}~\cite{B01} prove that $(H,C_k)$ is Ramsey infinite for each cycle $C_k$ if $H$ is $2$-connected and contains no induced cycles of length at least $\ell$, provided $k\geq \ell\geq 4$.
We think that all Ramsey finite pairs of graphs are characterized in Theorem~\ref{thm:Faudree}, see the discussion of the asymmetric case in Section~\ref{sec:conclusion}.


As for unordered graphs, we call a pair $(H,H')$ of ordered graphs \emph{Ramsey finite} if there are only finitely many minimal graphs in $\ors(H,H')$, and \emph{Ramsey infinite} otherwise.
Here an ordered graph $F\in\ors(H,H')$ is \emph{minimal} if $F'\not\in\ors(H,H')$ for each proper ordered subgraph $F'$ of $F$.
In case $H=H'$ we call $H$ itself Ramsey finite or infinite, respectively.
We shall establish results similar to Theorems~\ref{thm:UnorderedCycleInf} and~\ref{thm:unorderedForestCycle}\ref{enum:unordredMatchingFinite} while our results for ordered forests show that the ordering plays a significant role.

For wide families of ordered graphs we shall show that for a carefully chosen probability large random ordered graphs contain large minimal ordered Ramsey graphs.
Note that random graphs $G(n,p)$ are usually defined with vertex set $[n]$ and can be considered as ordered graphs without modifications.
R\"odl and Ruci\'nski~\cite{RR95} (see also~\cite{Gugel17}) determine for a given (unordered) graph $H$ the threshold probability that $G(n,p)$ is a Ramsey graph of $H$, provided that $H$ is not a forest whose components are stars or copies of $P_4$ only.
They prove that for such a graph there are constants $c$ and $c'$ such that if $p>cn^{-1/m_2(H)}$ then the probability that $G(n,p)$ is a Ramsey graph of $H$ tends to $1$ as $n\to\infty$ (the $1$-statement), while if $p<c'n^{-1/m_2(H)}$ then the probability that $G(n,p)$ is a Ramsey graph of $H$ tends to $0$ as $n\to\infty$ (the $0$-statement).
Note that the $0$-statement immediately carries over to ordered graphs by Observation~\ref{obs:basicRelOrderedRamsey}.
We prove a corresponding $1$-statement for ordered graphs.
Our proof closely follows a recent new proof of the $1$-statement for (unordered) graphs due to Nenadov and Steger~\cite{NS16} using the hypergraph container method.
It would be interesting whether one can deduce the $1$-statement for ordered graphs from the unordered $1$-statement directly, without reformulating the proof.
\begin{theorem}\label{thm:ordRandomRamsey}
 Let $H$ be an ordered graph that is not a partial matching. There is a constant $c$ such that  if $p>cn^{-1/m_2(H)}$ then the probability that $G(n,p)$ is an ordered Ramsey graph of $H$ tends to $1$ as $n\to\infty$.
\end{theorem}
%
%
%
Based on this theorem we use random graphs to prove the following result.
\begin{theorem}\label{thm:ordDensityRamseyInfinite}
 Let $H$ be an ordered graph.
 If $m(F)>m_2(H)$ for each ordered graph $F\in\ors(H)$, then $H$ is Ramsey infinite.
\end{theorem}
Essentially this theorem yields that every Ramsey finite ordered graph $H$ has a pseudoforest as a Ramsey graph.
Indeed by Theorems~\ref{thm:avgDegRamsey} and~\ref{thm:ordDensityRamseyInfinite} we have that $H$ is a forest and there is some $F\in\ors(H)$ with $m(F)\leq m_2(H)\leq 1$, that is, $F$ is a pseudoforest.
Using our results on sparse ordered Ramsey graphs from the first part of this article we obtain the following.
\begin{theorem}\label{thm:randomApplied}
If $H$ is a Ramsey finite ordered graph, then each component of $H$ which is not a monotone $P_3$ is a right star or each such component is a left star.
\end{theorem}
%
%
A full characterization of Ramsey finite ordered graphs which are connected is given in Corollary~\ref{cor:connectedFinite} below.
As in the unordered setting, we believe that Theorems~\ref{thm:ordRandomRamsey} and~\ref{thm:ordDensityRamseyInfinite} generalize to the asymmetric case, see the discussion in Section~\ref{sec:conclusion}.
Now we turn to pairs of ordered graphs where one is a forest and the other contains a cycle.
In the unordered setting such a pair is Ramsey finite if and only if the forest is a matching~\cite{B78, Lu94} (see Theorem~\ref{thm:unorderedForestCycle} above).
The following theorem gives a partial result for ordered graphs, similar to Theorem~\ref{thm:unorderedForestCycle}\ref{enum:unordredMatchingFinite}.
Recall that $I\subseteq V(G)$ is an interval of an ordered graph $G$ if for any $u$, $v\in I$ and $z\in V(G)$ with $u\leq z\leq v$ we have $z\in I$.
 An ordered graph $G$ with at least two vertices is \emph{loosely connected} if for any any partition $V_1\dot\cup V_2=V(G)$ of the vertices of $G$ into two disjoint intervals $V_1$ and $V_2$ there is an edge with one endpoint in $V_1$ and the other endpoint in $V_2$.
 See Figure~\ref{fig:segConnected}.
\begin{figure}
 \centering
 \includegraphics{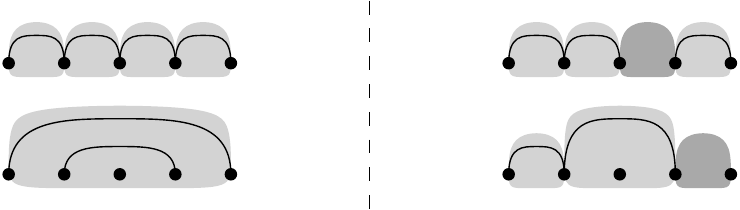}
 \caption{Each ordered graph on the left side is loosely connected while each ordered graph on the right is not loosely connected. We see that a disconnected ordered graph might be loosely connected, and each ordered graph which is not loosely connected contains two vertices next two each other which are not ``spanned'' by any edge (the dark gray parts).}
 \label{fig:segConnected}
\end{figure}
 Further $G\sqcup G'$ denotes the \emph{intervally disjoint union} of ordered graphs $G$ and $G'$, that is, a vertex disjoint union of $G$ and $G'$ where all vertices of $G$ are to the left of all vertices of $G'$.
 Note that each ordered graph $G$ without isolated vertices has a unique representation $G=G_1\sqcup\cdots\sqcup G_t$ where $G_i$ is a loosely connected ordered graph, $1\leq i\leq t$.
\begin{theorem}\label{thm:intervalUnionFinite}
 Let $s$ and $t$ be positive integers and let $H_1,\ldots,H_s$, $H'_1,\ldots,H'_t$ be loosely connected ordered graphs.
 If $(H_i,H'_j)$ is Ramsey finite for all $i\in[s]$, $j\in [t]$, then $(H_1\sqcup\cdots\sqcup H_s,H'_1\sqcup\cdots\sqcup H'_t)$ is Ramsey finite.
\end{theorem}
A \emph{monotone matching} is an ordered matching of the form $K_2\sqcup\cdots\sqcup K_2$.
Clearly $(H,K_2)$ is Ramsey finite for any ordered graph $H$.
\begin{corollary}\label{cor:monMatching}
 If $H'$ is a monotone matching, then $(H,H')$ is Ramsey finite for each ordered graph~$H$ without isolated vertices.
\end{corollary}
Finally we consider pairs of ordered forests.
A large part of the full characterization in the unordered setting (see Theorem~\ref{thm:Faudree}) is due to Ne\v{s}et\v{r}il and R\"odl~\cite{NR78} who prove that each pair of (unordered) forests which are not star forests is Ramsey infinite.
Their proof is based on the fact that each pair of (unordered) forests has Ramsey graphs of arbitrarily large girth.
This in turn relies on the fact that for each (unordered) forest $H$ there is an integer $k$ such that each graph of chromatic number at least $k$ contains a copy of $H$.
This second fact is not true for ordered forests~\cite{ForbiddenOrderedSubgraphs}.
We think though that the first fact holds for ordered graphs as well.
\begin{conjecture}\label{conj:orderedRamseyHighGirth}
 For each integer $t$ and any pair $(H,H')$ of ordered forests there is $F\in\ors(H,H')$ with $\girth(F)\geq t$.
\end{conjecture}
If Conjecture~\ref{conj:orderedRamseyHighGirth} is true, then each pair of ordered forests where $\ors(H,H')$ does not contain a forest has minimal Ramsey graphs of arbitrarily large (but finite) girth, and hence of arbitrarily large order.
\begin{observation}\label{obs:orderedInfHighGirth}
 Let $(H,H')$ be a pair of ordered forests such that $\ord(H,H')\geq 1$ and for each integer $t$ there is $F\in\ors(H,H')$ with $\girth(F)\geq t$.
 Then $(H,H')$ is Ramsey infinite.
\end{observation}
Here we focus on pairs of ordered forests $H$ which satisfy the second fact mentioned above, that is, there is an integer $k$ such that each graph of chromatic number at least $k$ contains a copy of~$H$.
We call such an ordered forest \emph{$\chi$-unavoidable}.
Ordered forest which are not $\chi$-unavoidable are discovered in~\cite{ForbiddenOrderedSubgraphs} and for some small such forests we show that they are Ramsey infinite in~\cite{RollinDiss}.
The proof from~\cite{NR78} can be easily adopted for $\chi$-unavoidable ordered forests.
So Conjecture~\ref{conj:orderedRamseyHighGirth} holds for $\chi$-unavoidable ordered forests and we have the following theorem.
\begin{theorem}\label{thm:unavoidNoRamseyForest}
 If $H$ and $H'$ are $\chi$-unavoidable ordered graphs and  $\ord(H,H')\geq 1$, then $(H,H')$ is Ramsey infinite.
\end{theorem}
 This theorem leaves to consider pairs of $\chi$-unavoidable ordered graphs with  $\ord(H,H')< 1$.
 We address such pairs of connected ordered graphs next and defer the study of such disconnected forests to future work.
 Recall that a right caterpillar is an ordered tree with segments being right stars with at least one edge each.
 Further if $S_i\preceq\cdots\preceq S_1$ are the segments of a right caterpillar $H$, then the \emph{defining sequence} of $H$ is $|E(S_1)|,\ldots,|E(S_i)|$.
 A left or right caterpillar with defining sequence $d_1,\ldots,d_i$ is called \emph{almost increasing} if $i\leq 2$ or ($i\geq 3$, $d_1\leq d_3$, and $d_2\leq\cdots\leq d_i$).
 See Figure~\ref{fig:incrCaterpillar}.
\begin{figure}
 \centering
 \includegraphics{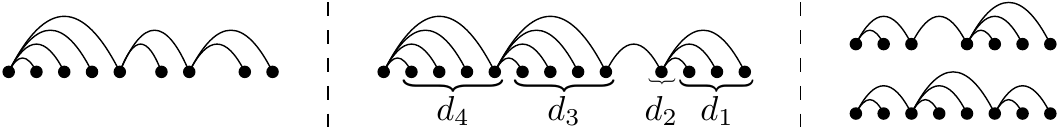}
 \caption{Two almost increasing right caterpillars (left, middle) and two not almost increasing right caterpillars (right).}
 \label{fig:incrCaterpillar}
\end{figure}
\begin{theorem}\label{thm:UnavoidInfinite}
 Let $(H,H')$ be  a Ramsey finite pair of $\chi$-unavoidable connected ordered graphs with at least two edges.
 Then $(H,H')$ is a pair of a right star and an almost increasing right caterpillar or a pair of a left star and an almost increasing left caterpillar.
\end{theorem}
\begin{theorem}\label{thm:CaterpillarFinite}
 Let $(H,H')$ be a pair of a right star and a right caterpillar or a pair of a left star and a left caterpillar, and let $d_1,\ldots,d_i$ be the defining sequence of the caterpillar.
 If either $i\leq 2$ or $d_1\leq\cdots\leq d_i$, then $(H,H')$ is Ramsey finite.
\end{theorem}
Unfortunately we do not resolve this case completely, see Conjecture~\ref{conj:unavoidableConnected} and the preceding discussion in Section~\ref{sec:conclusion}.
Nevertheless Theorems~\ref{thm:OrderedRamseyForest}, \ref{thm:randomApplied}, \ref{thm:unavoidNoRamseyForest} and~\ref{thm:CaterpillarFinite} yield the following result.
\begin{corollary}\label{cor:connectedFinite}
 A connected ordered graph is Ramsey finite if and only if it is a left or a right star.
\end{corollary}
A summary of our results is given in Table~\ref{fig:tableOrderedSummary}.
\begin{table}
\centering
\includegraphics{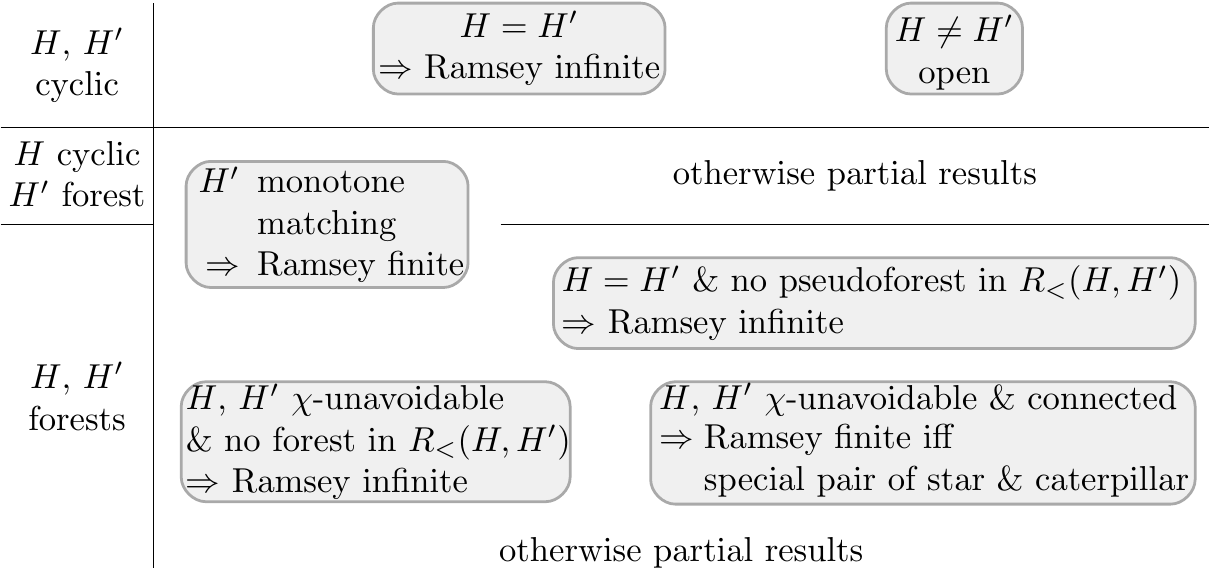}
\caption{Summary of results on Ramsey finiteness of $(H,H')$ for ordered graphs $H$ and $H'$.}
\label{fig:tableOrderedSummary}
\end{table}
%

\paragraph{Ordered Ramsey Numbers.}
Recently, Ramsey numbers were studied for ordered (hyper)graphs.
The \emph{ordered Ramsey number} of some ordered $r$-uniform hypergraph $H$ is the smallest integer $n$ such that for any $2$-coloring of the (hyper)edges of an ordered complete $r$-uniform hypergraph on $n$ vertices there is a copy of $H$ with all edges of the same color.
Due to several applications, mostly geometric Erd\H{o}s-Szekeres type results, Ramsey numbers of so-called monotone (hyper)paths received particular attention~\cite{Mub17}.
For given positive integers $\ell$ and $r$ a \emph{monotone $r$-uniform $\ell$-hyperpath} is an ordered $r$-uniform hypergraph with edges $E_1,\ldots,E_t$, where each edge forms an interval in the vertex ordering and $E_i\cap E_{i+1}$ consists of the $\ell$ rightmost vertices in $E_i$ and the  $\ell$ leftmost vertices in $E_{i+1}$, $i\in[t-1]$.
Building on previous results of Moshkovitz and Shapira~\cite{MS14}, Cox and Stolee~\cite{CS16} prove that the ordered Ramsey number of such paths $P$ grows, as a function in the number of edges of $P$, like a tower of height proportional to the maximum degree of $P$ (note that the maximum degrees of all sufficiently large monotone $r$-uniform $\ell$-hyperpaths coincide for fixed $\ell$ and $r$).
In contrast to this, the Ramsey numbers of unordered hyperpaths, and more general of any hypergraph of bounded maximum degree, are linear in the size.
Indeed for any uniformity $r$ and any positive integer $d$ there is a constant $c(r,d)$ such that for each (unordered) $r$-uniform hypergraph $H$ on $n$ vertices and of maximum degree at most $d$ its Ramsey number is at most $c(r,d)\,n$~\cite{RamseyMaxDegree, SparseHypergraphRamsey}.
In a similarly striking contrast to this result, Conlon~\textit{et al.}~\cite{ConlonFoxLeeSudakov} and independently Balko~\textit{et al.}~\cite{BalkoCibulkaKralKyncl} prove the existence of ordered matchings with superpolynomial Ramsey numbers.
On the other hand Conlon~\textit{et al.}~\cite{ConlonFoxLeeSudakov} present results showing that for dense graphs the ordered Ramsey numbers behave similar to the unordered Ramsey numbers.

\section{Proofs}\label{sec:proofs}


\subsection{Proof of Theorems~\ref{thm:OrderedRamseyForest} and~\ref{thm:OrderedRamseyPseudoforest}}\label{sec:proofRamseyForest}

First we introduce several types of edge-colorings which we shall use to proof that some ordered forest or pseudoforest is not a Ramsey graph for certain pairs of ordered graphs.
%
%
%
%
%
%
The \emph{distance} of an edge $e$ and a vertex $u$ in some graph $F$ is the smallest number of edges in a path that contains $u$ and $e$.
Three vertices $x<y<z$ of an ordered graph form a \emph{bend} if either $z$ is adjacent to $x$ and $y$ or $x$ is adjacent to $y$ and $z$.
An edge-coloring of an ordered graph $F$ is a
\begin{itemize}[wide]
 \item[\emph{star-coloring}] with respect to $u\in V(F)$ if an edge is colored red if its distance to $u$ is odd and blue otherwise,
 
 \item[\emph{bipartite-coloring}] with respect to a partition $A\dot\cup B=V(F)$ if  an edge is colored red if its left endpoint is in $A$ and blue otherwise,
 
 \item[\emph{bend-coloring}] with respect to $u\in V(F)$ if $F$ is a tree and an edge $e$ is colored red if its right endpoint is $u$ or the edge next to $e$ on the (unique) path to $u$ exists and forms a bend with $e$, and $e$ is colored blue otherwise.  See Figure~\ref{fig:ZigZagCol} (left) for examples of such a coloring.
\end{itemize}
   \begin{figure}
    \centering
    \includegraphics{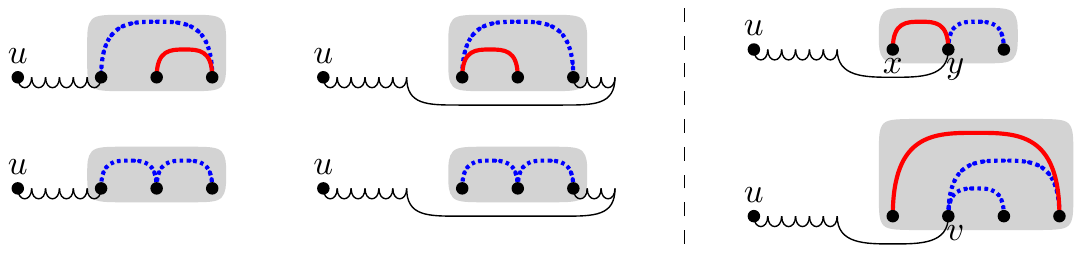}
    \caption{A bend-coloring of the edges of an ordered tree with respect to a vertex $u$ (left) in colors red (solid) and blue (dashed).
    There is no red monotone path on at least two edges (right, top) and in the blue component containing $v$ each vertex has at most one neighbor to the left (right, bottom).}
    \label{fig:ZigZagCol}
   \end{figure}

\begin{lemma}\label{lem:coloringFacts}
Let $F$ be an ordered graph and let $c$ be an edge-coloring of $F$.
\begin{enumerate}
 \item If $c$ is a star-coloring with respect to $u\in V(F)$, then each monochromatic component is a star and all edges incident to $u$ are red.\label{enum:starCol}
 
 \item If $c$ is a bipartite-coloring, then there is no monochromatic copy of a monotone $P_3$.\label{enum:bipCol}
 
 \item If $c$ is a bend coloring, then there is no red copy of a monotone $P_3$ and for each blue component $B$ either each vertex in $B$ has at most one neighbor to the left in $B$ or each vertex in $B$ has at most one neighbor to the right in $B$.\label{enum:bendCol}
\end{enumerate}
\end{lemma}
\begin{proof}
The first two statements follow immediately from the definitions, so we only prove~\ref{enum:bendCol}.
Suppose that $c$ is a bend-coloring of some ordered tree $F$ with respect to $u\in V(F)$.
First consider a copy $P$ of a monotone $P_3$ in $F$ that contains some red edge $xy$, $x<y$.
Then either $y=u$ or $xy$ forms a bend with the edge next to $xy$ on the path to $u$ in $F$.
In both cases the other edge in $P$ neither has $u$ as its right endpoint nor forms a bend with the edge next to it on the path to $u$ in $F$.
See Figure~\ref{fig:ZigZagCol} (right, top).
Hence there is no red monotone path on two edges.

Next consider a blue component $B$ and the vertex $v$ in $B$ that has shortest distance to $u$ in $F$ (it might happen that $u=v$).
Suppose that $v$ has some neighbor $w$ in $B$ with $v<w$.
If $u=v$, then each edge $w'v$ in $F$ with $w'<v$ is red (and $w'$ is not in $B$).
If $u\neq v$ and $v'v$ is the edge next to $vw$ on the path to $u$ in $F$, then $v'<v$ (as $vw$ is blue).
Hence $v$ does not have any neighbor $w'$ to the left in $B$, since any such edge $w'v$ is colored red as $w'v$ and $v'v$ form a bend.
Moreover each other vertex in $B$ has exactly one neighbor to the left in $B$, since otherwise there is a path from $v$ (and hence from $u$) to some vertex in $B$ that contains a bend and hence a red edge.
Hence each vertex in $B$ has at most one neighbor to the left in $B$.
See Figure~\ref{fig:ZigZagCol} (right, bottom).

Similar arguments show that if $v$ has some neighbor to the left in $B$, then each vertex in $B$ has at most one neighbor to the right in $B$.
\end{proof}

First we prove Lemma~\ref{lem:unorderedPseudo} on unordered graphs with Ramsey density at most $1$.
In that proof we shall freely use the star-coloring adopted for unordered graphs.
Clearly an analogous statement to Lemma~\ref{lem:coloringFacts}\ref{enum:starCol} holds in this case.

\begin{proof}[Proof of Lemma~\ref{lem:unorderedPseudo}]
 \begin{proofEnum}
  \item Let $H$ be a forest of maximum degree $d$ and let $H'$ be a star forest of maximum degree $t$.
  One can see that for any $2$-coloring of the edges of a sufficiently large $(d+t)$-ary tree $F$ without blue copies of $H'$ there is a copy of $H$ in some component of the red subgraph by a greedy embedding.
  Therefore $F\in\rs(H,H')$ and $r^m(H,H') < 1$.
  
  On the other hand consider a forest $F$ and a pair of graphs $(H,H')$. If $H$ is not a forest, then $F\not\in\rs(H,H')$ since coloring all its edges red yields neither red copies of $H$ nor blue copies of $H'$ (since $H'$ contains at least one edge).
  If neither $H$ nor $H'$ is a star forest then choose a root in each component of $F$ and color the edges of the components according to a star coloring with respect to their respective roots.
  Then there are neither red copies of $H$ nor blue copies of $H'$ by Lemma~\ref{lem:coloringFacts}\ref{enum:starCol} and $F\not\in\rs(H,H')$.
  So in both cases  there is no forest in $\rs(H,H')$ and we have that $r^m(H,H') \geq 1$.
  
  \item Consider a pair $(H,H')$ of graphs.
  If one of $H$ or $H'$ is not a pseudoforest, or one of $H$ or $H'$ contains a cycle while the other is not a partial matching (that is, contains a copy of $P_3$), then clearly there is no pseudoforest in $\rs(H,H')$.
  
  If $H$ is a partial matching and $H'$ is a proper pseudoforest, then let $F$ be a vertex disjoint union of $|V(H)|$ many copies of $H'$. 
  For any coloring of the edges of $F$ either all edges in one of the copies of $H'$ are blue or there is red copy of $H$.
  Hence $F$ is a Ramsey graph of $(H,H')$ and $r^m(H,H') \leq 1$.
  Moreover $r^m(H,H') \geq 1$ by part~\ref{enum:unorderedRamseyForest} and thus $r^m(H,H') = 1$.
  
  This leaves to consider pairs $(H,H')$ of ordered forests.
  If one of $H$ or $H'$ is a star forest, then $r^m(H,H') < 1$ by part~\ref{enum:unorderedRamseyForest}.
  So suppose that both $H$ and $H'$ contain a copy of $P_4$.
  
  First assume that $H'$ has a component which is not a star and not a $P_4$.
  Then $H$ contains either a copy of $P_5$ or a copy of a graph $P$ obtained from $P_4$ by adding a pendant edge, that is, by adding a new vertex $u$ and an edge connecting $u$ to one the vertices of degree $2$ in $P_4$.
  Consider some proper pseudoforest $F$.
  We shall prove that $F\not\in\rs(H,H')$.
  
  If $H'$ contains a copy of $P_5$, then let $F'$ denote the forest obtained from $F$ by removing a smallest set $E$ of edges which contains one edge from each cycle in $F$.
  Color each component of $F'$ according to some star coloring and color all edges in $E$ blue.
  Then by Lemma~\ref{lem:coloringFacts}\ref{enum:starCol} the red edges in $F$ form a star forest and there is no blue copy of $P_5$ in $F$.
  Hence there is no red copy of $H$ and no blue copy of $H'$ and $F$ is not a Ramsey graph of $(H,H')$.
  So there is no pseudoforest in $\rs(H,H')$ in this case, as $F$ was arbitrary.
  
  If $H'$ contains a copy of $P_4$ with a pendant edge, then let $C$ denote the subgraph of $F$ formed by all the cycles in $F$ and let $F'$ denote the subgraph of $F$ formed by all edges not in cycles.
  Each component of $F'$ contains at most one vertex from $C$.
  Color all edges in $C$ blue and color each component of $F'$ according to a star coloring with respect to the unique vertex shared with $C$, if it exists, and with respect to an arbitrary vertex otherwise.
  Then by Lemma~\ref{lem:coloringFacts}\ref{enum:starCol} the red edges form a star forest while the blue edges form a vertex disjoint union of a star forest and cycles.
  Hence there is no red copy of $H$ and no blue copy of $H'$ and $F\not\in\rs(H,H')$.
  Again there is no pseudoforest in $\rs(H,H')$ in this case, as $F$ was arbitrary.
  
  It remains to consider pairs $(H,H')$ of forests of stars and copies of $P_4$, both with at least one copy of $P_4$.
  Let $F$ be a graph obtained from a $5$-cycle with vertices $u_1,\ldots,u_5$ by adding vertices $v_1,\ldots,v_5$ and an edge $u_iv_i$ for each $i$, $1\leq i\leq 5$.
  Moreover let $d$ denote the largest degree among all vertices in $H$ and $H'$, let $F'$ be a forest that is a Ramsey graph for a pair of stars on $d$ edges, and let $F''$ be a forest that is a Ramsey graph for a pair of a star on $d$ edges and $P_4$.
  Such forests exist by part~\ref{enum:unorderedRamseyForest}.
  One can see that $F\to(P_4,P_4)$~\cite{MP13} and that a suitable vertex-disjoint union of several copies of the graphs $F$, $F'$, and $F''$ forms a Ramsey graph of $(H,H')$.
  Since such a union is a pseudoforest we have $r^m(H,H') \leq 1$.
  Moreover $r^m(H,H') \geq 1$ by part~\ref{enum:unorderedRamseyForest} and thus $r^m(H,H') = 1$.\qedhere  
 \end{proofEnum}
\end{proof}

\begin{lemma}\label{lem:noRamseyPseudoforest}
 Let $H$ and $H'$ be ordered graphs. Then $\ors(H,H')$ does not contain a pseudoforest in each of the following cases.
 \begin{enumerate}
  \item One of $H$ and $H'$ contains a cycle and the other is not a partial matching.\label{enum:properPseudo}
 
  \item One of $H$ or $H'$ contains a vertex with two neighbors to the right and the other contains a vertex with two neighbors to the left.\label{enum:sameOrientaionPseudo}
  
  \item One of $H$ and $H'$ contains a copy of a monotone $P_3$ and the other contains a copy of an ordered $P_4$.\label{enum:P4}
  
  \item One of $H$ and $H'$ contains a copy of a monotone $P_3$ and the other contains a copy of a star on three edges that is neither a right star nor a left star .\label{enum:threeStar}
 \end{enumerate}
\end{lemma}
\begin{proof}
 Let $F$ be a pseudoforest.
 For each of the cases we shall give a coloring of the edges of $F$ without red copies of $H$ or blue copies of $H'$.
 \begin{proofEnum} 
  \item Without loss of generality assume that $H$ contains a cycle and $H'$ is not a partial matching.
  Choose a smallest set $E$ of edges which  contains one edge from each of the cycles of $F$.
  Then color all edges in $E$ blue and all the other edges of $F$ red.
  Since the edges in $E(F)\setminus E$ form a forest there is no red copy  of $H'$ and since the edges in $E$ form a matching there is no blue copy  of $H'$.
  Hence $F\not\in\ors(H,H')$ and $\ors(H,H')$ contains no pseudoforest.
 
  \item Without loss of generality assume that $H$ contains a vertex with two neighbors to the right and $H'$ contains a vertex with two neighbors to the left.
  We give a $2$-coloring of the edges of $F$ without red copies of $H$ or blue copies of $H'$  by induction on the number of edges of $F$.
  Indeed such a coloring clearly exists if $\lvert E(F)\rvert =1$.
  If $\lvert E(F)\rvert >1$ we distinguish two cases.
  First suppose that $F$ has some vertex $v$ of degree $1$.
  Remove $v$ from $F$ and color the resulting pseudoforest inductively.
  If $v$ is to the left of its neighbor $u$ in $F$ then color $uv$ red, otherwise color it blue.
  This coloring of $F$ contains neither red copies of $H$ nor blue copies of $H'$.
    Next assume that $F$ contains no vertex of degree $1$, that is, $F$ is a vertex disjoint union of cycles.
  For each even cycle in $F$ color its edges alternatingly red and blue and for each odd cycle color both edges incident to its leftmost vertex blue and the remaining edges alternatingly red and blue.
  This coloring contains neither red copies of $H$ nor blue copies of $H'$. 
  In both cases $F\not\in\ors(H,H')$ and hence $\ors(H,H')$ contains no pseudoforest.
  
  \item Without loss of generality assume that $H$ contains a copy of a monotone $P_3$ and $H'$ contains a copy $P$ of an ordered $P_4$.
  We shall give a coloring of $F$ with no red copies of $P_3$ and no blue copies of $P$.
  Since  $P_3$ and $P$ are connected we assume without loss of generality that $F$ is connected.
  We distinguish cases based on the ordering of $P$.
  
  First assume that $P$ forms a monotone $P_4$.
  If $F$ is bipartite, then color its edges using a bipartite-coloring with respect to an arbitrary bipartition of $F$.
  Then there is no red copy of $H$ and no blue copy of $H'$ by Lemma~\ref{lem:coloringFacts}\ref{enum:bipCol}.
  If $F$ is not bipartite, then we obtain a bipartite graph $F'$ from $F$ by removing some edge $e$ from the (unique) odd cycle in $F$.
  Note that for any bipartition of $F'$ the endpoints of $e$ belong to the same part.
  Choose such a partition $A\cup B=V(F')$ such that the endpoints of $e$ belong to $A$.
  Color the edges in $F'$ using the bipartite-coloring with respect to the partition formed by $A$ and $B$.
  Further color $e$ blue.
  By Lemma~\ref{lem:coloringFacts}\ref{enum:bipCol} there is no red copy of $H$ and each blue copy of a monotone $P_3$ contains $e$.
  Since the left endpoint of $e$ is in $A$ we have that all edges incident to this vertex to the left are colored red.
  Therefore there is no blue copy of a monotone $P_4$ and thus no blue copy of $H'$.
  In both cases $F\not\in\ors(H,H')$ and hence $\ors(H,H')$ contains no pseudoforest.
  
  Next assume that $P$ contain a copy of a monotone $P_3$ whose rightmost vertex has two neighbors to the left in $P$.
  If $F$ is bipartite, then color its edges using a bipartite-coloring with respect to an arbitrary bipartition of $F$.
  Then there is no red copy of $H$ and no blue copy of $H'$ by Lemma~\ref{lem:coloringFacts}\ref{enum:bipCol}.
  Otherwise consider the odd cycle $C$ in $F$ and edges $uv$, $u'v$ in $C$ such that $v$ is the rightmost vertex of $C$ and $u'<u<v$.
  Let $T$ denote the subgraph of $F$ formed by the union of all monotone paths in $F$ whose leftmost vertex is $u$.
  Then $T$ is a tree, since it does not contain $u'v$ (and $C$ is the only cycle in $F$).
  Color all edges of $T$ blue.
  Then there is no blue copy of $P$ in $T$.
  The remaining edges form a forest $F'$ since $F'$ does not contain $uv$.
  Consider a bipartition $A\cup B=V(F')$ where each vertex shared with $T$ is in $B$.
  Such a partition exists since either $u$ and $v$ are connected by a path in $F'$ with an odd number of vertices or are in distinct components.
  Color the edges in $F'$ using a bipartite-coloring with respect to the partition formed by $A$ and $B$, see Figure~\ref{fig:treeColorings} (left).
  Then in $F'$ there are no monochromatic copies of a monotone $P_3$ by Lemma~\ref{lem:coloringFacts}\ref{enum:bipCol}.
  In particular there is no red copy of $H$ in $F$, since all edges in $T$ are blue.
  Moreover each blue edge in $F'$ does not contain any vertex of $T$ since edges in $F'$ may share only their right endpoint with $T$.
  Hence each blue component of $F$ is either in $F'$ or in $T$ and hence there is no blue copy of $H'$ in $F$.
  Altogether $F\not\in\ors(H,H')$ and hence $\ors(H,H')$ contains no pseudoforest.

\begin{figure}
 \centering
 \includegraphics{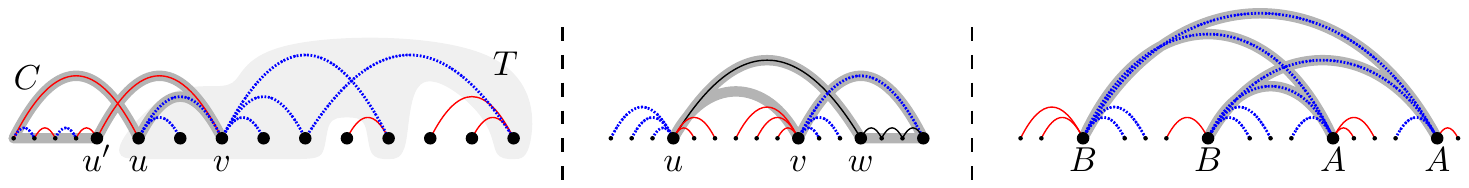}
 \caption{Colorings of proper pseudoforests. In the left all edges of a certain tree $T$ are colored blue and the remaining edges are coloring using a bipartite coloring. In the middle an edge $uv$ is removed from a cycle and the remaining forest is colored with a bend coloring with respect to $v$. In the right all edges of a cycle are colored blue and the remaining edges are colored using a bipartite coloring.}
 \label{fig:treeColorings}
\end{figure}

  Next assume that $P$ contain a copy of a monotone $P_3$ whose leftmost vertex has two neighbors to the right.
  In this case $\ors(H,H')$ contains no pseudoforest with arguments similar to the previous case.
    
  Finally assume that $P$ does not contain any copy of a monotone $P_3$.
  Then $P$ has some vertex with two neighbors to the left and another vertex with two neighbors to the right.
  If $F$ is a tree, then color its edges according to some bend-coloring with respect to an arbitrary vertex.
  Then there is no red copy of a monotone $P_3$ and no blue copy of $P$ by Lemma~\ref{lem:coloringFacts}\ref{enum:bendCol}.
  Otherwise let $u$ denote the leftmost vertex of the cycle in $F$ and let $v$ and $w$ be its neighbors in that cycle.
  Let $F'$ be a tree obtained y removing the edge $uv$ from $F$.
  Color $F'$ with a bend-coloring with respect to $v$, see Figure~\ref{fig:treeColorings} (middle).
  Then in $F'$ there is no red copy of a monotone $P_3$ and no blue copy of $P$ by Lemma~\ref{lem:coloringFacts}\ref{enum:bendCol}.
  Note that any edge $xu$ in $F'$ with $x<u$ is colored blue, since the next edge on the path to $u'v$ is $uw$ which does not form a bend with $xu$.
  Moreover any edge $vy$ in $F'$ with $v<y$ is colored blue.
  The coloring of $F'$ gives a coloring of $F$ by coloring $uv$ red.
  Then there is no blue copy of $P$ since there is no such copy in $F'$ and no red copy of a monotone $P_3$ since edges $xu$, $x<u$, and edges $vy$, $v<y$, are blue.
  Altogether $F\not\in\ors(H,H')$ and hence $\ors(H,H')$ contains no pseudoforest.
  
  \item  Without loss of generality assume that $H$ contains a copy of a monotone $P_3$ and $H'$ contains a star on three edges whose center has one neighbor to the left and two neighbors to the right.
  Note that $H'$ contains a copy of a monotone $P_3$.
  If $F$ is bipartite, then color its edges using a bipartite-coloring with respect to an arbitrary bipartition of $F$.
  Then there is no red copy of $H$ and no blue copy of $H'$ by Lemma~\ref{lem:coloringFacts}\ref{enum:bipCol}.
  Otherwise color all edges of $F$ that are contained in a cycle blue.
  The remaining edges form a forest $F'$.
  Choose a bipartition $A\cup B=V(F')$ such that all vertices in the cycle in $F$ that have two neighbors to the right in that cycle are in $B$ and the other vertices of the cycle are in $A$.
  Such a partition exists since the vertices of the cycle in $F$ are in different components of $F'$.
  Color the edges in $F'$ using a bipartite-coloring with respect to the partition formed by $A$ and $B$, see Figure~\ref{fig:treeColorings} (right).
  Then in $F'$ there are no monochromatic copies of a monotone $P_3$ by Lemma~\ref{lem:coloringFacts}\ref{enum:bipCol}.
  In particular there is no red copy of $H$ in $F$, since all edges not in $F'$ are blue.
  Moreover each vertex which is left endpoint of at least two blue edges in $F$ is in $B$.
  Such a vertex is not a right endpoint of any blue edge.
  This shows that there is no blue copy of $H'$ in $F$.
 Altogether $F\not\in\ors(H,H')$ and hence $\ors(H,H')$ contains no pseudoforest.\qedhere
 \end{proofEnum}
\end{proof}

\begin{lemma}\label{lem:noRamseyForest}
  Let $H$ and $H'$ be ordered forests. Then $\ors(H,H')$ does not contain a forest in each of the following cases.
  \begin{enumerate} 
   \item Both $H$ and $H'$ contain a component that is not a star.\label{enum:star}
   
   \item One of $H$ or $H'$ contains a vertex with two neighbors to the right and the other contains a vertex with two neighbors to the left.\label{enum:sameOrientaion}
   
   \item Both $H$ and $H'$ contain a monotone path on two edges.\label{enum:monotonePath}
   
   \item One of $H$ and $H'$ contains a copy of a monotone $P_3$ and the other contains a copy of an ordered $P_4$.\label{enum:alternatingPath}
  \end{enumerate}
 \end{lemma}
 \begin{proof}
  Let $F$ be a forest.
  For each of the cases we shall give a coloring of the edges of $F$ without red copies of $H$ or blue copies of $H'$.
  \begin{proofEnum} 
   \item For each component of $F$ color its edges according to a star-coloring with respect to some arbitrary vertex in that component.
   Then each color class forms a star forest by Lemma~\ref{lem:coloringFacts}\ref{enum:starCol}, that is, there is neither a monochromatic copy of $H$ nor of $H'$.
   Hence $\ors(H,H')$ contains no forest.
   
   \item This follows immediately from Lemma~\ref{lem:noRamseyPseudoforest}\ref{enum:sameOrientaionPseudo}.
   
   \item Color the edges of $F$ according to a bipartite-coloring with respect to an arbitrary bipartition.   
   This coloring contains neither red copies of $H$ nor blue copies of $H'$ by Lemma~\ref{lem:coloringFacts}\ref{enum:bipCol}.
   Hence $\ors(H,H')$ contains no forest.
   
   \item This follows immediately from Lemma~\ref{lem:noRamseyPseudoforest}\ref{enum:P4}.\qedhere
%
  \end{proofEnum}\end{proof}

\begin{proof}[Proof of Theorem~\ref{thm:OrderedRamseyForest}]
 We shall prove that all pairs of ordered forests that do not have a forest as an ordered Ramsey graph are covered by Lemma~\ref{lem:noRamseyForest}.
 To this end we provide explicit constructions of ordered forests that are ordered Ramsey graphs for the remaining pairs.
 
 First we shall show that each pair $(H,H')$ where $\ors(H,H')$ contains a forest satisfies at least one of the cases of this theorem.
 Let $(H,H')$ be such a pair.
 Clearly $H$ and $H'$ are forests, since any monochromatic subgraph of an edge-colored forest is a forest itself.
 If either $H$ or $H'$ is a partial matching then Case~\ref{enum:matching} of this theorem holds.
 So assume that neither $H$ nor $H'$ is a partial matching.
 Due to Lemma~\ref{lem:noRamseyForest}~\ref{enum:star} one of $H$ or $H'$ is a star forest.
 Without loss of generality assume that $H$ is a star forest.
 Due to Lemma~\ref{lem:noRamseyForest}~\ref{enum:monotonePath} one of $H$ or $H'$ does not contain a monotone $P_3$.
 
 First suppose that $H$ does not contain a monotone $P_3$.
 Then each component of $H$ is a left or a right star.
 Due to Lemma~\ref{lem:noRamseyForest}~\ref{enum:sameOrientaion} the following holds.
 If each component of $H$ is a right star, then each vertex of $H'$ has at most one neighbor to the left (as $H$ is not a partial matching).
 Thus $H$ and $H'$ satisfy Case~\ref{enum:rightForests} of this theorem.
 Similarly, if each component of $H$ is a left star, then each vertex of $H'$ has at most one neighbor to the right.
 Thus $H$ and $H'$ satisfy Case~\ref{enum:leftForests} of this theorem.
 If $H$ contains a right star on two edges as well as a left star on two edges, then each component of $H'$ is a monotone path.
 Thus $H$ and $H'$ satisfy Case~\ref{enum:RightLeftStars} of this theorem.
 
 Now suppose that $H$ contains a monotone $P_3$.
 Then $H'$ neither contains a monotone $P_3$ nor any ordered $P_4$ due to Lemma~\ref{lem:noRamseyForest}~\ref{enum:monotonePath} and~\ref{enum:alternatingPath}.
 Therefore each component of $H'$ is a left or a right star.
 The same arguments as above show that $H$ and $H'$ satisfy one of the Cases~\ref{enum:rightForests}, \ref{enum:leftForests}, or~\ref{enum:RightLeftStars} of this theorem.
 

 Next consider two ordered forests $H$ and $H'$ that satisfy Case~\ref{enum:matching}, \ref{enum:rightForests}, \ref{enum:leftForests}, or~\ref{enum:RightLeftStars} of this theorem.
 We shall show that there is a forest in $\ors(H,H')$ and distinguish which case of this theorem holds.
 First of all suppose that $H$ and $H'$ together contain some $t>0$ isolated vertices.
 Let $\bar{H}$ and $\bar{H'}$ be obtained from $H$ respectively $H'$ by removing these $t$ isolated vertices.
 Then there is a forest in $\ors(H,H')$ if and only if there is a forest in $\ors(\bar{H},\bar{H'})$.
 Indeed, if $F$ is an ordered forest in $\ors(H,H')$, then $F\in\ors(\bar{H},\bar{H'})$.
 Suppose that $\bar{F}$ is an ordered forest in $\ors(\bar{H},\bar{H'})$.
 Then we obtain an ordered forest in $\ors(H,H')$ by adding $t$ isolated vertices to the left of all vertices in $\bar{F}$, to the right of all vertices in $\bar{F}$, as well as between any pair of consecutive vertices of $\bar{F}$.
 Similarly $\ors(H,H')$ contains a partial matching if and only if $\ors(\bar{H},\bar{H'})$ contains a partial matching.
 For the remaining proof we assume that neither $H$ nor $H'$ contains isolated vertices.

 \begin{proofEnum} 
  \item Without loss of generality assume that $H$ is a  matching.
  Consider a complete ordered graph $K$ of order $r=\orn(H,H')$ with vertices $v_1<\cdots<v_r$.  
  Let $k=\lvert V(H')\rvert $, $m'=\binom{r}{k}$, and $m=\binom{r-1}{k-1}$.
  Note that $K$ contains exactly $m'$ copies of $H'$ and each vertex of $K$ is contained in exactly $m$ copies of $H'$ in $K$.
  We shall construct an ordered graph $F$ that is a vertex disjoint union of $m'$ copies of $H'$.
  For each $i\in[r]$ let $H_i^1,\ldots,H_i^m$ denote the copies of $H'$ in $K$ containing $v_i$.
  Choose disjoint ordered vertex sets $V_i=(v_i^1,\ldots,v_i^m)$ of size $m$ each, $i\in[r]$.
  Let $F$ denote the ordered graph with vertex set $\cup_{i=1}^r V_i$, $V_1\prec\cdots\prec V_r$, where $v_i^jv_s^t$ is an edge in $F$ if and only if $H_i^j=H_s^t$ and the edge $v_iv_s$ is in $H_i^j=H_s^t$, $1\leq i<s\leq r$, $1\leq j,t\leq m$.
   See Figure~\ref{fig:matchingRamseyForest}.
   \begin{figure}
    \centering
    \includegraphics{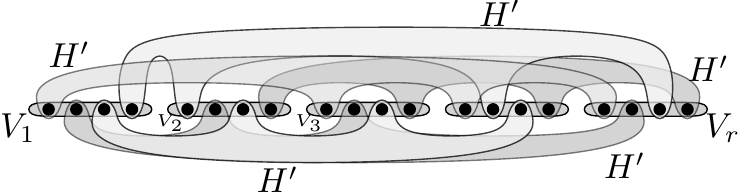}
    \caption[An ordered graph in $\ors(H,H')$ for some matching~$H$.]{Disjoint copies of $H'$ forming an ordered graph in $\ors(H,H')$ for some matching~$H$.}
    \label{fig:matchingRamseyForest}
   \end{figure}
   
   Observe that $F$ is a vertex disjoint union of copies of $H'$ and hence a forest.
   Moreover, if $H'$ is a matching, then $F$ is a matching.
   We claim that $F\in\ors(H,H')$.
   Consider a $2$-coloring $c$ of the edges of $F$.
   We shall show that there is either a red copy of $H$ or a blue copy of $H'$.
   To this end consider the edge-coloring $c'$ of $K$ where an edge $v_iv_s$, $1\leq i<s\leq r$, is colored red if there is at least one red edge between $V_i$ and $V_s$ in $F$ and blue otherwise.
   Due to the choice of $K$ there is either a red copy of $H$ or a blue copy of $H'$ under $c'$.
   In either case there is a red copy of $H$ respectively a blue copy of $H'$ under $c$.
   Thus $F\in\ors(H,H')$.

   \item Without loss of generality assume that each component of $H$ is a right star and each vertex in $H'$ has at most one neighbor to the left.
   We shall prove that there is a forest in $\ors(H,H')$ by induction on the size of $H'$.
   If $H'$ has only one edge, then clearly $H$ is in $\ors(H,H')$ and we are done.
   So suppose that $H'$ has at least two edges.
   Let $w$ denote the rightmost vertex of $H'$ and let $F'$ denote an ordered forest in $\ors(H,H'-w)$ which exists by induction (note that $w$ is a leaf in $H'$ and $H'$ has no isolated vertices).
   We shall construct an ordered forest $F$ as follows.
   Let $v_1<\cdots<v_n$ denote the vertices of $H$ and let $S$ denote the set of left endpoints of edges in $H$, that is, the centers of the right stars in $H$.
   
   Choose $n$ disjoint ordered sets of vertices $V_1\prec\cdots\prec V_n$ of size $\lvert V(F')\rvert $ each.
   For each $i\in[n]$, with $v_i\in S$, add edges among the vertices in $V_i$ such that $V_i$ induces a copy of $F'$.
   For each edge $v_iv_j$ in $H$, $1\leq i<j\leq n$, add an arbitrary perfect matching between $V_i$ and $V_j$.
   See Figure~\ref{fig:blowupStarForest} for an illustration.
   \begin{figure}
    \centering
    \includegraphics{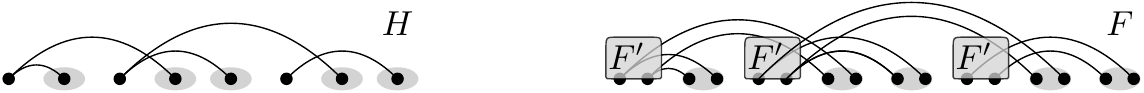}
    \caption[A forest in $\ors(H,H')$ where $H$ is a forest of right stars and $H'$ is a forest where each vertex has at most one neighbor to the left.]{A forest $F$ in $\ors(H,H')$ formed from disjoint copies of a forest $F'\in\ors(H,H'-w)$ with attached leaves. Here $H$ is a forest of right stars, $H'$ is a forest where each vertex has at most one neighbor to the left, and $w$ is the rightmost vertex of $H'$.}
    \label{fig:blowupStarForest}
   \end{figure}

   Then $F$ is a vertex disjoint union of copies of $F'$ with some leaves attached.
   In particular $F$ is a forest.
   We claim that $F\in\ors(H,H')$.
   Consider a coloring of the edges of $F$.
   For each $u\in V_i$ there is an edge $uv$ with $v\in V_j$ for some $j>i$, since each $v_i\in S$ is left endpoint of some edge $v_iv_j$ in $H$.
   If for each $v_i\in S$ there exists $u\in V_i$ such that an edge $uv$ is red whenever $v\in V_j$ with $j>i$, then there is a red copy of $H$.
   So assume that there is some $v_i\in S$ such that for each $u\in V_i$ there is a blue edge $uv$ for some $v\in V_j$, $j>i$.
   Since $V_i$ induces a copy of $F'$, there is either a red copy of $H$ or a blue copy of $H'-w$.
   In the latter case some edge between $V_i$ and $V_j$ yields a blue copy of $H'$ in $F$, since $w$ is rightmost and of degree $1$ in $H'$.
   Altogether $F$ is a forest in $\ors(H,H')$.
   
   \item This follows from Case~\ref{enum:rightForests}.
   
   \item Without loss of generality assume that each component of $H$ is a left or a right star and each component of $H'$ is a monotone path.
   We shall prove that there is a forest in $\ors(H,H')$ by induction on the number of components of $H$ and the size of $H'$.
   If $H$ has only one component, then there is a forest in $\ors(H,H')$ by Case~\ref{enum:rightForests} or~\ref{enum:leftForests}.
   If $H'$ has only one edge, then $H$ is in $\ors(H,H')$ and we are done.
   Suppose that $H$ has at least two components and $H'$ has at least two edges.
   Let $S$ denote a component of $H$.
   Without loss of generality assume that $S$ is a right star.
   Let $w$ denote the rightmost vertex of $H'$ (recall that $H'$ has no isolated vertices so $w$ is of degree $1$).
   By induction there are ordered forests $A\in\ors(H-S,H')$ and $B\in\ors(H,H'-w)$.
   Let $a_1<\cdots<a_n$ denote the vertices of $A$.
   Let $F'$ consist of an intervally disjoint union of $n+1$ copies $B_1,\ldots,B_{n+1}$ of $B$.
   Consider an ordered forest $F''$ that is a vertex disjoint union of $F'$ and $A$ where for each $i$, $1\leq i\leq n$, the vertex $a_i\in V(A)$ is between $B_i$ and $B_{i+1}$.
   See Figure~\ref{fig:spannedStars} (left).
   Obtain an ordered forest $F$ from $F''$ as follows.
   For all $i$, $j$ with $1\leq i\leq j\leq n+1$, and each vertex $u$ in $B_i$ add a star with center $u$ and $\lvert E(S)\rvert $ leaves to the right of $B_j$ and, if $j<n+1$, to the left of $a_j$, such that these leaves are distinct for distinct pairs $i$, $j$, and vertices $u$.
   See Figure~\ref{fig:spannedStars} (right) for an illustration.
   \begin{figure}
    \centering
    \includegraphics{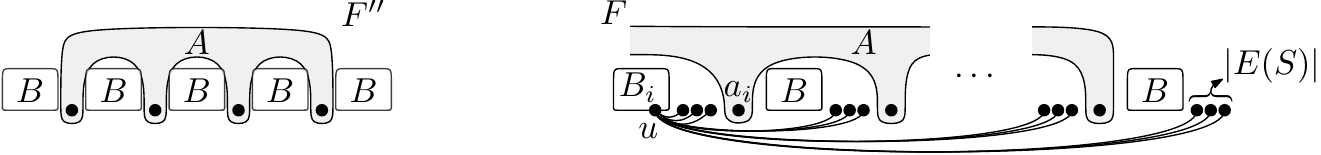}
    \caption[A forest in $\ors(H,H')$ where $H$ is a forest of left and right stars and $H'$ is a forest of monotone paths.]{Forests $A\in\ors(H-S,H')$ and $B\in\ors(H,H'-w)$ forming a forest $F$ in $\ors(H,H')$.
    Here $S$ is a component of $H$ and $w$ is the rightmost vertex in $H'$, where each component of $H$ is a left or a right star and each component of $H'$ is a monotone path.}
    \label{fig:spannedStars}
   \end{figure}
   
   Clearly $F$ is a forest since $F''$ is a forest and all the leaves added in the last step are distinct.
   We claim that $F\in\ors(H,H')$.
   Consider a $2$-coloring of the edges of $F$ without blue copies of $H'$.
   Then there is a red copy $K$ of $H-S$ in $A$.
   Consider some $i\in[n+1]$ such that the position of any vertex $u$ in $B_i$ in $K+u$ corresponds to the position of the center of $S$ in $H$.
   If for some $u\in V(B_i)$ all the edges $uv$ in $F$ where $v$ is to the right of $B_i$ are red, then this red right star together with $K$ contains a red copy of $H$.
   So suppose that for each vertex $u\in V(B_i)$ there is a blue edge $uv$ where $v$ is to the right of $B_i$.
   Then there is no blue copy of $H'-w$ in $B_i$, as otherwise there is a blue copy of $H'$ since $w$ is rightmost and of degree $1$ in $H'$.
   Hence there is a red copy of $H$ in $B_i$.
   Altogether $F\in\ors(H,H')$.
 \end{proofEnum}
 
 Finally we prove the last statement of the theorem (that is, $\ors(H,H')$ contains a partial matching if and only if both $H$ and $H'$ are partial matchings).
 In the proof of Case~\ref{enum:matching} we give a partial matching in $\ors(H,H')$ provided that both $H$ and $H'$ are partial matchings.
 The other way round there is clearly no partial matching in $\ors(H,H')$ provided that one of $H$ or $H'$ is not a partial matching.
\end{proof}

\begin{proof}[Proof of Theorem~\ref{thm:OrderedRamseyPseudoforest}]
  Similar to the proof of Theorem~\ref{thm:OrderedRamseyForest} we shall prove that all pairs of connected ordered graphs that do not have a pseudoforest as an ordered Ramsey graph are covered by Lemma~\ref{lem:noRamseyPseudoforest}.
  For those remaining pairs which do not have a forest as a Ramsey graph by Theorem~\ref{thm:OrderedRamseyForest} we provide explicit constructions of ordered pseudoforests that are ordered Ramsey graphs.
  First we shall show that any pair $(H,H')$ of connected ordered graphs with $\ord(H,H')=1$ where $\ors(H,H')$ contains a pseudoforest is a pair of $K_2$ and a connected ordered proper pseudoforest or both $H$ and $H'$ form a monotone $P_3$.

 First suppose that $H'=K_2$.
 Clearly $H\in\ors(H,K_2)$ is the unique minimal ordered Ramsey graph of $(H,K_2)$.
 Hence $\ord(H,K_2)=1$ and $\ors(H,K_2)$ contains a pseudoforest if and only if $H$ is a proper pseudoforest.
  
 Next let $H$ and $H'$ be connected ordered graphs with $\ord(H,H')=1$ where $\ors(H,H')$ contains a pseudoforest and  suppose that both $H$ and $H'$ contain at least two edges.
 We shall show that both $H$ and $H'$ form a monotone $P_3$.
 Since $H$ and $H'$ are connected they are not matchings.
 Therefore both $H$ and $H'$ are forests due to Lemma~\ref{lem:noRamseyPseudoforest}\ref{enum:properPseudo}.
 If neither $H$ nor $H'$ contains a copy of a monotone $P_3$, then either there is no pseudoforest in $\ors(H,H')$ by Lemma~\ref{lem:noRamseyPseudoforest}\ref{enum:sameOrientaionPseudo} or $\ord(H,H')<1$ by Theorem~\ref{thm:OrderedRamseyForest}\ref{enum:rightForests} or \ref{enum:leftForests}.
 So assume, without loss of generality, that $H$ contains a copy of a monotone $P_3$.
 Then $H'$ does not contain any copy of an ordered $P_4$ by Lemma~\ref{lem:noRamseyPseudoforest}\ref{enum:P4}, that is, $H'$ is a star (since it is connected).
 If $H'$ is a left or a right star then either $H$ contains a left or right star in the reverse direction and there is no pseudoforest in $\ors(H,H')$ by Lemma~\ref{lem:noRamseyPseudoforest}\ref{enum:sameOrientaionPseudo} or $\ord(H,H')<1$ by Theorem~\ref{thm:OrderedRamseyForest}\ref{enum:rightForests} or \ref{enum:leftForests}.
 Hence $H'$ is neither a left or a right star and therefore $H'$ forms a monotone $P_3$, since otherwise there is no pseudoforest in $\ors(H,H')$ by Lemma~\ref{lem:noRamseyPseudoforest}\ref{enum:threeStar}.
 Now the same arguments applied with roles of $H$ and $H'$ switched show that both $H$ and $H'$ form a copy of a monotone $P_3$.
 
 It remains to prove that $\ord(H,H')= 1$ in this case.
 By Theorem~\ref{thm:OrderedRamseyForest} (see also Lemma~\ref{lem:noRamseyForest}\ref{enum:monotonePath}) we have that $\ord(H,H')\geq 1$.
 It remains to give a pseudoforest in $\ors(H,H')$.
 Let $F$ denote the ordered graph obtained from a monotone path $u_1u_2u_3u_4u_5$ by adding an edge $u_2u_4$.
 Clearly $F$ is a proper pseudoforest and we claim that $F\in\ors(H,H')$. 
 Consider a $2$-coloring of the edges of $F$.
 If there are two incident edges on the path $u_1u_2u_3u_4u_5$ of the same color then these form either a red copy of $H$ or a blue copy of $H'$.
 Otherwise the edges $u_1u_2$ and $u_4u_5$ are of different colors.
 Then no matter which color is assigned to the edge $u_2u_4$, there is either a red copy of $H$ or a blue copy of $H'$.
 Altogether $F\in\ors(H,H')$ and thus $\ord(H,H')= 1$.
 \end{proof}

\subsection{Proof of Theorems~\ref{thm:ordRandomRamsey},~\ref{thm:ordDensityRamseyInfinite} and~\ref{thm:randomApplied}}\label{sec:proofOrderedCycle}

We shall use the hypergraph container method due to Saxton and Thomason~\cite{ST15} and independently Balogh~\textit{et al.}~\cite{Containers15} to prove Theorem~\ref{thm:ordRandomRamsey}.
We will follow the arguments from~\cite{NS16} (see also~\cite{KFBook}) using the following result of Saxton and Thomason~\cite{ST16}.
Considering ordered graphs instead of (unordered) graphs only affects the involved constants.
Recall that the \emph{density} of a graph is $m(G) =\max\{\lvert E(G')\rvert/\lvert V(G')\rvert\mid G'\subseteq G\}$ and that the \emph{$2$-density} is $m_2(G) =\max\{(\lvert E(G')\rvert-1)/(\lvert V(G')\rvert-2)\mid G'\subseteq G, \lvert V(G')\rvert\geq 3, \lvert E(G')\rvert\geq 1\}$.
For a hypergraph $\cH$ and some integer $\ell$ let $\Delta_\ell(\cH) =\max\{\lvert \cE\rvert \mid \cE\subseteq E(\cH), \lvert \cap_{E\in\cE} E\rvert \geq\ell\}$.
\begin{theorem}[Cor.~1.3~\cite{ST16}]\label{thm:hyperContainerST}
 For all $r\in\bN$ and for any $\epsilon>0$ there is $C>0$ such that for all $r$-uniform hypergraphs $\cH$ with average degree $d>0$ and each $\tau$, $0<\tau\leq 1$, with $\Delta_\ell(\cH)\leq Cd\tau^{\ell-1}$, $2\leq \ell\leq r$, the following holds.
 There is a function $f:2^{V(\cH)}\to 2^{V(\cH)}$ such that for each independent set $I$ in $\cH$ there is $S\subseteq V(\cH)$ with
 \begin{enumerate} 
  \item $S\subseteq I\subseteq f(S)$,
  \item $\lvert S\rvert \leq \tau \lvert V(\cH)\rvert $,
  \item $\lvert E(\cH[f(S)])\rvert  \leq \epsilon \lvert E(\cH)\rvert $.
 \end{enumerate}
\end{theorem}
We shall also use a corollary to the so-called Small-Subgraph-Theorem of Erd\H{o}s and R\'enyi~\cite{ErdRenyi} and Bollob{\'a}s~\cite{BollSmallSubgraphs}.
\begin{theorem}[\cite{BollSmallSubgraphs},\cite{ErdRenyi}]\label{thm:smallSubgraphRandom}
 Let $H$ be a graph with at least one edge. The probability that a random graph $G(n,p)$ contains a copy of $H$ tends to $0$ (as $n\to\infty$) if $pn^{1/m(H)}\to 0$ (as $n\to\infty$).
\end{theorem}

The following lemma is an analog of Corollary 2.2. from~\cite{NS16}.
We include a proof for ordered graphs for completeness.
\begin{lemma}\label{lem:superRamsey2}
 Let $H$ be an ordered graph, $\sigma=\orn(H,H,H)$, $\epsilon=\tfrac{1}{4}\sigma^{-\sigma}$, $\delta=\lvert E(H)\rvert \tfrac{1}{2}\sigma^{-\sigma}$, and $n\geq \sigma$.
 If $E_1$, $E_2\subseteq E(K_n)$ with $E_i$ inducing at most $\epsilon n^{\lvert V(H)\rvert }$ copies of $H$, $i\in[2]$, then $\lvert E_1\cup E_2\rvert \leq (1-\delta) \binom{n}{2}$.
\end{lemma}
\begin{proof}
 Let $t=\lvert V(H)\rvert $, $E=E_1\cup E_2$, and $E'=E(K_n)\setminus E$.
 Then for each set of $\sigma$ vertices in $K_n$ there is a copy of $H$ with all edges in $E_1$, all edges in $E_2$, or all edges in $E'$.
 Note that each copy of $H$ in $K_n$ is contained in at most $n^{\sigma-t}$ such $\sigma$-sets.
 Thus there are at least $\tbinom{n}{\sigma} n^{t-\sigma}$ copies of $H$ in $K_n$ with all edges in one of $E_1$, $E_2$, or $E'$.
 Therefore $E'$ induces at least $\tbinom{n}{\sigma} n^{t-\sigma} - 2\epsilon n^{t} \geq (\sigma^{-\sigma}-2\epsilon) n^t = \tfrac{1}{2}\sigma^{-\sigma} n^t$ copies of $H$.
 Since each edge in $E'$ is contained in at most $n^{t-2}$ copies of $H$, there are at least $\lvert E(H)\rvert \tfrac{1}{2}\sigma^{-\sigma} n^t / n^{t-2} = \lvert E(H)\rvert \tfrac{1}{2}\sigma^{-\sigma} n^2\geq \delta \binom{n}{2}$ edges in $E'$.
 Thus $\lvert E\rvert \leq (1-\delta)\binom{n}{2}$.
\end{proof}
%
%
The following lemma is implicitly contained in~\cite{NS16} for unordered graphs.
\begin{lemma}[Ramsey Containers~\cite{NS16}]\label{lem:orderedRamseyCont}
 Let $H$ be an ordered graph without isolated vertices.
 Then there are constants $N_0$, $C'$, $\delta>0$, and a function $g$ mapping ordered graphs to ordered graphs such that for each $n$, $n\geq N_0$, and each $F\not\in\ors(H)$ on vertex set $[n]$  there is an ordered graph $P$ with
 \begin{enumerate} 
  \item $V(P)$, $V(g(P))\subseteq[n]$,
  \item $P\subseteq F\subseteq g(P)$,
  \item $\lvert E(P)\rvert \leq C' n^{2-1/m_2(H)}$,
  \item $\lvert E(g(P))\rvert  \leq (1-\delta)\binom{n}{2}$.
 \end{enumerate} 
\end{lemma}
\begin{proof}
 Consider some (large) $n$ and an ordered complete graph $K$ on vertex set $[n]$.
 Let $\cH=\cH_n$ be a hypergraph with vertex set $E(K)$ that contains an edge $E\subseteq E(K)$ if and only if $E$ forms a copy of $H$ in $K$.
 Observe that there is a $1$-$1$ correspondence between subsets of $V(\cH)$ and ordered graphs on vertex set $[n]$ (without isolated vertices).
 Let $t=\lvert V(H)\rvert $, $r=\lvert E(H)\rvert $.
 Then $\cH$ is $r$-uniform, $\lvert V(\cH)\rvert =\binom{n}{2}$, $\lvert E(\cH)\rvert =\binom{n}{t}$, and its average degree is $d(\cH)=\frac{r\lvert E(\cH)\rvert }{\lvert V(\cH)\rvert }$.
 Let $\delta$ and $\epsilon$ be given by Lemma~\ref{lem:superRamsey2} and let $C\leq 1$ be given by Theorem~\ref{thm:hyperContainerST} for $r$ and $\epsilon$ as defined here.
 Further let $\tau = \frac{t^t}{C} n^{-1/m_2(H)}$ and let $v(\ell)=\min\{\lvert V(H')\rvert \mid H'\subseteq H, \lvert E(H')\rvert =\ell\}$, for $\ell\in[r]$.
 Finally choose $N_0\geq \max\{\orn(H),t\}$ such that $\frac{t^t}{C} N_0^{-1/m_2(H)}\leq1$ and consider $n\geq N_0$.
 Then $d(\cH)>0$ and for each $\ell$, $2\leq \ell\leq r$,
 \begin{align}
  m_2(H) &= \max_{\substack{H'\subseteq H\\\lvert V(H')\rvert \geq 3}}\tfrac{\lvert E(H')\rvert -1}{\lvert V(H')\rvert -2} \geq \max_{H'\subseteq H,\lvert E(H')\rvert =\ell}\tfrac{\ell-1}{\lvert V(H')\rvert -2} = \tfrac{\ell-1}{v(\ell)-2},\label{eq:m2}\\
  \Delta_\ell(\cH) &\leq \tbinom{n-v(\ell)}{t-v(\ell)} \leq n^{t-v(\ell)} \overset{(\ref{eq:m2})}{\leq} n^{t-2-\frac{\ell-1}{m_2(H)}} \leq  \frac{t^t\binom{n}{t}}{\binom{n}{2}} n^{-\frac{\ell-1}{m_2(H)}}\nonumber\\
  &\leq \tfrac{r\lvert E(\cH)\rvert }{\lvert V(\cH)\rvert } C \left(\tfrac{t^t}{C}\right)^{\ell-1}n^{-\frac{\ell-1}{m_2(H)}} =  C d \tau^{\ell-1}.\nonumber
 \end{align}
 Due to Theorem~\ref{thm:hyperContainerST} there is a function $f:2^{V(\cH)} \to 2^{V(\cH)}$ such that for each independent set $I$ of $\cH$ there is $S\subseteq V(\cH)$ with $S\subseteq I\subseteq f(S)$, $\lvert S\rvert \leq \tau \lvert V(\cH)\rvert $, and $\lvert E(\cH[f(S)])\rvert  \leq \epsilon \lvert E(\cH)\rvert $.
 Now consider some $F\not\in\ors(H)$ on $n$ vertices and a $2$-coloring of the edges of $F$ with no monochromatic copy of $H$.
 Then the color classes form independent sets $I_1$ and $I_2$ in $\cH$ (since $H$ has no isolated vertices).
 As argued above there are sets $S_1$, $S_2\subseteq V(\cH)$ with $S_i\subseteq I_i\subseteq f(S_i)$, $\lvert S_i\rvert \leq \tau \lvert V(\cH)\rvert  \leq \frac{t^t}{C} n^{2-1/m_2(H)}$, and $\lvert E(\cH[f(S_i)])\rvert  \leq \epsilon \lvert E(\cH)\rvert  \leq \epsilon n^t$, $i=1$, $2$.
 Due to the latter condition and Lemma~\ref{lem:superRamsey2} we have $\lvert f(S_1)\cup f(S_2)\rvert \leq (1-\delta)n^2$.
 The statement of the lemma follows with $C'=2\frac{t^t}{C}$, $P$ being the graph formed by $S_1\cup S_2$, and $g(P)$ being the graph formed by $f(S_1)\cup f(S_2)$.
\end{proof}

\begin{proof}[Proof of Theorem~\ref{thm:ordRandomRamsey}]
	Let $H$ be an ordered graph which is not a partial matching, that is, $m_2(H)>\frac{1}{2}$.
	We shall show that there is a constant $c$ such that with $p=cn^{-1/m_2(H)}$ a random graph $G(n,p)$  is an ordered Ramsey graph of $H$ with probability tending to $1$ as $n$ tends to infinity.	
	First suppose that $H$ has no isolated vertices.
	Let $N_0$, $C'$, $\delta>0$ be constants and let $g$ be a function given by Lemma~\ref{lem:orderedRamseyCont}.
  Let $\cP$ denote the set of ordered graphs $P$ with $V(P)\subseteq [n]$, $\lvert E(P)\rvert \leq C' n^{2-1/m_2(H)}$, and no isolated vertices.
  Note that we can assume that $\lvert E(g(P))\rvert \leq (1-\delta)\binom{n}{2}$ for each $P\in\cP$.  
  For sufficiently large $c$ the probability that there is some $P\in\cP$ with $P\subseteq G(n,p)\subseteq g(P)$ is at most
  \begin{align}
  \sum_{P\in\cP} P(P\subseteq G(n,p)\subseteq g(P)) &\leq \quad\sum_{P\in\cP} p^{\lvert E(P)\rvert } (1-p)^{\delta\binom{n}{2}}\nonumber\\[5pt]
  &=\sum_{i=0}^{\floor{C' n^{2-1/m_2(H)}}}\binom{\binom{n}{2}}{i} p^i (1-p)^{\delta\binom{n}{2}} \label{eq:cycleInf1}\\[5pt]
  &\leq\sum_{i=0}^{\floor{C' n^{2-1/m_2(H)}}} \left(\tfrac{en^2p}{i}\right)^i e^{-p\delta\binom{n}{2}}\label{eq:cycleInf2}\\[5pt]
  &\leq C'n^2 \left(\tfrac{en^2p}{C' n^{2-1/m_2(H)}}\right)^{C' n^{2-1/m_2(H)}} e^{-p\delta\binom{n}{2}}\label{eq:cycleInf3}\\[5pt]
  &\leq C'n^2 \left(\tfrac{ec}{C'}\right)^{C' n^{2-1/m_2(H)}} e^{-(c\delta/4) n^{2-1/m_2(H)}}\nonumber\\[5pt]
  &=C'n^2 e^{(\overbrace{\scriptstyle C'+C'\ln(c)-C'\ln(C')-(c\delta/4)}^{<0}) n^{\overbrace{\scriptscriptstyle 2-1/m_2(H)}^{>0}}}\nonumber\\[5pt]
  &\longrightarrow 0\enskip (n\to \infty).\nonumber
  \end{align}
  Here equality~\ref{eq:cycleInf1} holds since for each $i$ there are $\binom{\binom{n}{2}}{i}$ graphs on $i$ edges in $\cP$.
  Inequality~\ref{eq:cycleInf2} holds since $\binom{\binom{n}{2}}{i}\leq (en^2/i)^i$.
  Finally inequality~\ref{eq:cycleInf3} holds since for any fixed $r>0$ the function $(r/x)^x$ of $x$ is increasing for $0<x\leq r$ (note that its derivative is $(r/x)^x(\ln(r/x)-1)$) and $C' n^{2-1/m_2(H)} \leq ecn^{2-1/m_2(H)} = en^2p$.
  If $n\geq N_0$ and $G(n,p)\not\in\ors(H)$, then by Lemma~\ref{lem:orderedRamseyCont} there is some $P\in\cP$ with $P\subseteq G(n,p)\subseteq g(P)$.
  Thus the calculation above shows that $G(n,p)\in\ors(H)$ with probability tending $1$ as $n\to\infty$.
  
  Now suppose that $H$ contains $a$ isolated vertices and let $H'$ be obtained from $H$ by removing all isolated vertices.
  Note that $m_2(H)=m_2(H')$.
  Let $c=c(H')$ denote a constant such that with $p=cn^{-1/m_2(H)}$ we have $G(n,p)\in\ors(H')$ with probability tending $1$ as $n\to\infty$ which exists as argued before.
  Further let $c'=(a+2)^{1/m_2(H)}c$, $p'=c'n^{-1/m_2(H)}$, and consider a random graph $F=G(n,p')$ with vertex set $[n]$.
  We claim that $F\in\ors(H)$ with probability tending $1$ as $n\to\infty$.
  Indeed, consider $U(n)=\{(a+1)k\mid 1\leq k\leq \frac{n-a}{a+1}\}\subseteq [n]$ and the subgraph $F'$ of $F$ with vertex set $U(n)$.
  Let $\tilde{n}=|U(n)|=\lfloor (n-a)/(a+1)\rfloor\geq n/(a+2)$.
  Then $p'\geq c'((a+2)\tilde{n})^{-1/m_2(H)} = c\tilde{n}^{-1/m_2(H)}$.
  So $F'$ behaves like $G(\tilde{n},c\tilde{n}^{-1/m_2(H)})$ and therefore $F'\in\ors(H')$ with probability tending $1$ as $n\to\infty$ by the choice of $c$.
  Finally observe that there are $a$ vertices in $F$ to the left of $U(n)$, $a$ vertices to the right of $U(n)$, and $a$ vertices between any pair of vertices from $U(n)$.
  So any monochromatic copy of $H'$ in $F'$ yields a monochromatic copy of $H$ in $F$.
  Therefore $F\in\ors(H)$ with probability tending $1$ as $n\to\infty$.
\end{proof}

\begin{proof}[Proof of Theorem~\ref{thm:ordDensityRamseyInfinite}]
 Let $H$ be an ordered graph such that $m(F)>m_2(H)$ for each $F\in\ors(H)$.
 Then $H$ is not a partial matching by Theorem~\ref{thm:OrderedRamseyForest}.
 Let $c$ denote the constant given by Theorem~\ref{thm:ordRandomRamsey}, that is, with $p=cn^{-1/m_2(H)}$ a random graph $G(n,p)$  is an ordered Ramsey graph of $H$ with probability tending to $1$ as $n$ tends to infinity.
 Let  $p=cn^{-1/m_2(H)}$, $F=G(n,p)$, and fix some integer $t>0$.
 We shall prove that $F\in\ors(H)$ and $F[V]\not\in\ors(H)$ for each $t$-subset $V\subseteq V(F)$ with probability tending to $1$ as $n$ tends to infinity.
 Hence there is a minimal ordered Ramsey graph contained in $F$ on more than $t$ vertices.
 Since $t$ is arbitrary there are minimal ordered Ramsey graphs of $H$ of arbitrarily large order and hence $H$ is Ramsey infinite.

 Consider $n$ sufficiently large such that $p\leq 1$.
 Let $\uo{F}$ denote the underlying (unordered) graph of $F$, and let $\cF$ denote the set of all ordered graphs in $\ors(H)$ with $t$ vertices.
 We show that with high probability $\uo{F}$ does not contain the underlying graph of any member of $\cF$ as a subgraph.
 Consider some fixed $F'\in\cF$.
 We have $m(F') > m_2(H)$ by assumption.
 Therefore \[pn^{1/m(F')}=n^{1/m(F')-1/m_2(H)}\longrightarrow 0\quad (n\to\infty).\]
 So with high probability $\uo{F}$ does not contain the underlying graph of $F'$ as a subgraph by Theorem~\ref{thm:smallSubgraphRandom}.
 In particular $F$ does not contain $F'$ as an ordered subgraph.
 This shows that with probability tending to $1$ (as $n\to\infty$) $F$ does not contain any member of $\cF$ as a subgraph, since $\cF$ is a finite set.
 So with probability tending to~$1$ (as $n\to\infty$) $F[V]\not\in\ors(H)$ for each $t$-subset $V\subseteq V(F)$.

  Altogether we see that for sufficiently large $n$ there is an ordered graph $F$ with $F\in\ors(H)$ and $F[V]\not\in\ors(H)$ for each $t$-subset $V\subseteq V(F)$.
  Therefore there are arbitrarily large minimal ordered Ramsey graphs of $H$ and hence $H$ is Ramsey infinite.
\end{proof}

\begin{proof}[Proof of Theorem~\ref{thm:randomApplied}]
 Let $H$ be a Ramsey finite ordered graph.
 Then there is $F\in\ors(H)$ with $m(F)\leq m_2(H)$ by Theorem~\ref{thm:ordDensityRamseyInfinite}.
 Therefore $m_2(H)\leq 1$ by Theorem~\ref{thm:avgDegRamsey} and thus $m(F)\leq 1$, that is, $H$ is a forest  and $F$ is a pseudoforest.
 Hence $H$ does not contain any ordered copy of $P_4$ by Lemma~\ref{lem:noRamseyPseudoforest}~\ref{enum:sameOrientaionPseudo} and~\ref{enum:P4}, that is, $H$ is a star forest.
 Moreover each star on three edges in $H$ is either a right star or a left star by Lemma~\ref{lem:noRamseyPseudoforest}~\ref{enum:threeStar}.
 Finally consider the components of $H$ with at least two edges that are not monotone paths.
 Then either all these components are left stars or all these components are right stars by Lemma~\ref{lem:noRamseyPseudoforest}~\ref{enum:sameOrientaionPseudo}. 
\end{proof}



\subsection{Proof of Theorem~\ref{thm:intervalUnionFinite}}\label{sec:intervalUnion}
 Let $s$ and $t$ be positive integers and let $H_1,\ldots,H_s$, $H'_1,\ldots,H'_t$ be loosely connected ordered graphs such that $(H_i,H'_j)$ is Ramsey finite for all $i\in[s]$, $j\in[t]$.
 In particular neither of $H_1,\ldots,H_s$, $H'_1,\ldots,H'_t$ is an isolated vertex.
 Further let $H=H_1\sqcup\cdots\sqcup H_s$ and $H'=H'_1\sqcup\cdots\sqcup H'_t$.
 In order to prove that $(H,H')$ is Ramsey finite we shall show that each minimal ordered Ramsey graph of $(H,H')$ is a member of the following finite family of ordered graphs. 
 Let $\cF_s^t$ denote the set of all ordered graphs that are isomorphic to a (not necessarily disjoint) union of ordered graphs $F_i^j$, $i\in[s]$, $j\in[t]$, where $F_i^j$ is a minimal ordered Ramsey graph of $(H_i,H_j)$, and for each $i\in[s]$, $j\in[t-1]$, we have $F_i^j\prec F_{i}^{j+1}$, and for each $i\in[s-1]$, $j\in[t]$, we have $F_i^j\prec F_{i+1}^{j}$.
 See Figure~\ref{fig:segmentUnion} for an illustration in the case $s=t=3$.
 For some graph $A\in \cF_s^t$ that is isomorphic to such a union of graphs $F_i^j$, $i\in[s]$, $j\in[t]$, let $A_i^j$ denote the copy of $F_i^j$ in $A$.
 \begin{figure}
  \centering
  \includegraphics{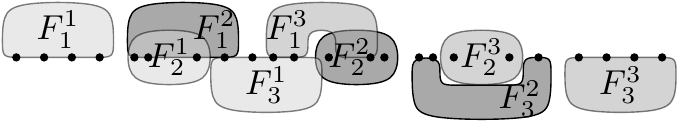}
  \caption[A Ramsey graph for $(H_1\sqcup H_2\sqcup H_3,H'_1\sqcup H'_2\sqcup H'_3)$.]{A Ramsey graph for $(H_1\sqcup H_2\sqcup H_3,H'_1\sqcup H'_2\sqcup H'_3)$ where $F_i^j\in\ors(H_i,H'_j)$, $i$, $j=1$, $2$, $3$. The subgraphs of the same color correspond to constant $j$, the subgraphs on the same horizontal layer (above, on, respectively below the line of vertices) correspond to constant $i$.}
  \label{fig:segmentUnion}
 \end{figure}
 Note that $\cF_s^t$ is finite.
  We claim that $F\in\ors(H,H')$ if and only if $F$ contains some member of $\cF_s^t$.
 
 \medskip
 
 Consider a $2$-coloring of the edges of some $A\in\cF_s^t$.
 We shall prove that there is a red copy of $H$ or a blue copy of $H'$ by induction on $s$ and $t$.
 If $s=1$, then $A=A_1^1\sqcup\cdots\sqcup A_1^t$ where $A_1^j\in\ors(H,H'_j)$ for each $j\in[t]$.
 If $t=1$, then $A=A_1^1\sqcup\cdots\sqcup A_s^1$ where $A_i^1\in\ors(H_i,H')$ for each $i\in[s]$.
 In both cases it is easy to see that there is either a red copy of $H$ or a blue copy of $H'$.
 Hence $A\in\ors(H,H')$.
 Suppose that $s$, $t>1$.
 Let $A'$ denote the subgraph of $A$ formed by all subgraphs $A_i^j$ with $(i,j)\neq (s,t)$.
 Then $A'$ contains some member of $\cF_{s-1}^t$ and some member of $\cF_{s}^{t-1}$.
 By induction, $A$ is in $\ors(H,H'_1\sqcup\cdots\sqcup H'_{t-1})$ and in $\ors(H_1\sqcup\cdots\sqcup H_{s-1},H')$.
 If there is no red copy of $H$ and no blue copy of $H'$ in $A'$, then there is a red copy of $H_1\sqcup\cdots\sqcup H_{s-1}$ and a blue copy of $H'_1\sqcup\cdots\sqcup H'_{t-1}$.
 Observe that $A'\prec A_s^t$ in $A$.
 Since $A_s^t\in\ors(H_s,H'_t)$ there is a red copy of $H$ or a blue copy of $H'$ in either case.
 Thus $A\in\ors(H,H')$.
 
 \medskip
 
 Now consider an ordered graph $F$ that does not contain any member of $\cF_s^t$.
 We shall prove that $F\not\in\ors(H,H')$ by induction on $s$ and $t$.
 Consider the case $s=1$.
 If $t=1$ then $F$ does not contain any minimal ordered Ramsey graph of $(H_1,H'_1)=(H,H')$.
 Clearly $F\not\in\ors(H,H')$.
 Suppose that $t>1$ and let $p$ denote the rightmost vertex in $F$ such that $\{q\in V(F)\mid p\leq q\}$ induces a copy of some graph from $\ors(H,H'_t)$ in $F$.
 Let $F_\ell$ and $F_r$ denote the subgraphs of $F$ induced by all vertices strictly to the left respectively strictly to the right of $p$.
 Then $F_\ell$ does not contain any member of $\cF_1^{t-1}$.
 By induction on $t$ there is a coloring of the edges of $F_\ell$ without red copies of $H$ or blue copies of $H'_1\sqcup\cdots\sqcup H'_{t-1}$.
 Moreover there is a coloring of $F_r$ without red copies of $H$ or blue copies of $H'_{t}$.
 Color all remaining edges blue (that is, all edges incident to $p$ and all edges between $F_\ell$ and $F_r$).
 Then there is no red copy of $H$ since $H=H_1$ is loosely connected (so in particular not an isolated vertex).
 Moreover each blue copy of $H'_t$ contains some vertex $q$ with $q\leq p$ and thus there is no blue copy of $H'$.
 This shows that $F\not\in\ors(H,H')$.
 If $t=1$ and $s>1$, then $F\not\in\ors(H,H')$ due to symmetric arguments.
 
 So suppose that $s$, $t>1$.
 If $F$ does not contain any member of $\cF_{s-1}^t$, then $F\not\in\ors(H_1\sqcup\cdots\sqcup H_{s-1},H')$ by induction on $s$.
 Hence $F\not\in\ors(H,H')$.
 So assume that $F$ contains some member of $\cF_{s-1}^t$.
 Let $A$ denote such a copy where for each $i\in[s-1]$ and $j\in[t]$ the rightmost vertex of $A_i^j$ is leftmost among all such copies.
 Note that we can simultaneously choose such leftmost vertices for all $i\in[s-1]$ and $j\in[t]$. 
 If $F$ does not contain any subgraph of the form $B^1\sqcup\cdots\sqcup B^t$ where $B^j\in\ors(H_s,H'_j)$ for each $j\in[t]$, then the arguments from case $s=1$ show that $F\not\in\ors(H_s,H')$ and thus $F\not\in\ors(H,H')$.
 Otherwise let $B=B^1\sqcup\cdots\sqcup B^t$ denote such a subgraph of $F$ where for each $j\in[t]$ the leftmost vertex of $B^j$ is rightmost among all such subgraphs.
 Again we can simultaneously choose such rightmost vertices for all $j\in[t]$. 
 Since $A\cup B\not\in\cF_s^t$ there is $j\in[t]$ such that $B^j$ is not to the right of $A_{s-1}^j$.
 See Figure~\ref{fig:goodColIntervalUnion} for an illustration.
 \begin{figure}
  \centering
  \includegraphics{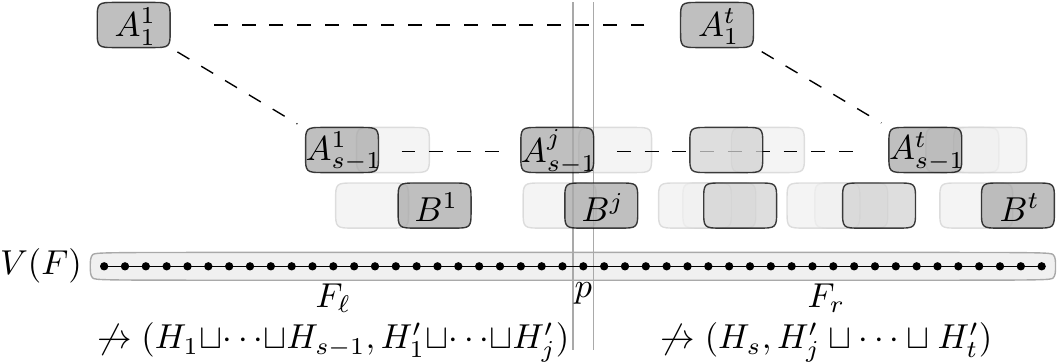}
  \caption[Illustration in proof of Theorem~\ref{thm:intervalUnionFinite}.]{An illustration of the proof of Theorem~\ref{thm:intervalUnionFinite}. Here $A_i^j\in\ors(H_i,H'_j)$, $B^j\in\ors(H_s,H'_j)$, and $F$ does not contain any element from $\cF_s^t$ . Therefore a combination of colorings of $F_\ell$ and $F_r$ shows that $F\not\to(H_1\sqcup\cdots\sqcup H_s,H'_1\sqcup\cdots\sqcup H'_t)$.}
  \label{fig:goodColIntervalUnion}
 \end{figure}
 Let $p$ denote the rightmost vertex of $A_{s-1}^j$ and let $F_\ell$ and $F_r$ denote the subgraphs of $F$ induced by all vertices strictly to the left respectively strictly to the right of $p$.
 By the choice of $A$ the graph $F_\ell$ does not contain any member of $\cF_{s-1}^j$.
 By the choice of $B$ the graph $F_r$ does not contain any subgraph of the form $B'_j\sqcup\cdots\sqcup B'_t$ where $B'_{j'}\in\ors(H_s,H_{j'})$, $j\leq j'\leq t$.
 Therefore there is a coloring of the edges of $F_\ell$ without red copies of $H_1\sqcup\cdots\sqcup H_{s-1}$ or blue copies of $H'_1\sqcup\cdots\sqcup H'_j$ by induction on $s$.
 Similarly there is a coloring of the edges of $F_r$ without red copies of $H_s$ or blue copies of $H'_j\sqcup\cdots\sqcup H'_t$ due to the arguments from case $s=1$.
 Color all remaining edges red (that is, all edges incident to $p$ and all edges between $F_\ell$ and $F_r$).
 Then each red copy of $H_1\sqcup\cdots\sqcup H_{s-1}$ contains some vertex $q$ with $p\leq q$.
 Hence there is no red copy of $H$.
 Similarly we see that there is no blue copy of $H'$ since all edges incident to $p$ and all edges between $F_\ell$ and $F_r$ are red (and $H'_j$ is not an isolated vertex).
 Altogether $F\not\in\ors(H,H')$.
 This shows that each minimal ordered Ramsey graph $F$ of $(H,H')$ contains some $F'\in\cF_s^t$.
 Since we proved in the beginning that $F'\in\ors(H,H')$, we have that $F=F'$ and $F\in\cF_s^t$.
 Thus $(H,H')$ is Ramsey finite.\qed

\subsection{Proof of Theorems~\ref{thm:unavoidNoRamseyForest} and~\ref{thm:UnavoidInfinite}}\label{sec:unavoidInf}


\begin{proof}[Proof of Theorem~\ref{thm:unavoidNoRamseyForest}]
 As $H$ and $H'$ are $\chi$-unavoidable there is an integer $k$ such that each ordered graph of chromatic number at least $k$ contains a copy of $H$ and a copy of $H'$.
 For each $t\geq 3$ let $G_t$ be an ordered graph of chromatic number at least $k^2$ and girth at least $t$.
 First we prove that $G_t\in\ors(H,H')$.
 Consider a $2$-coloring of the edges of $G_t$.
 Since $\chi(G_t)\geq k^2$ one of the color classes forms an ordered graph of chromatic number at least $k$.
 Therefore this subgraph contains a (monochromatic) copy of both of $H$ and $H'$.
 Thus $G_t\in\ors(H,H')$.
  
 For each $t\geq 3$ let $G_t'$ be a minimal ordered Ramsey graph of $(H,H')$ that is a subgraph of $G_t$.
 For each $t\geq 3$ the graph $G_t'$ contains a cycle, since $\ors(H,H')$ contains no forest.
 As $G_t$ has girth at least $t$, infinitely many of the graphs $G_t'$ are not isomorphic.
 Thus $(H,H')$ is Ramsey infinite.
\end{proof}


For the proof of Theorem~\ref{thm:UnavoidInfinite} it remains to consider pairs of $\chi$-unavoidable connected ordered graphs that are not covered by Theorem~\ref{thm:unavoidNoRamseyForest}, i.e., that have a forest as an ordered Ramsey graph.
Recall that all pairs of ordered graphs having a forest as a Ramsey graph are characterized by Theorem~\ref{thm:OrderedRamseyForest}.
We give explicit constructions of infinitely many ordered Ramsey graphs for these pairs using so-called determiners as building blocks.
We introduce determiners and give explicit constructions of such ordered graphs next.
The concept of (unordered) determiners is used by Burr~\textit{et al.}~\cite{BNR85} to construct Ramsey graphs with certain properties.

 Let $\vec{S}_p$ denote a right star with $p$ edges.
 In the proof we shall frequently use the following unions of ordered graphs $G$ and $G'$.
 Recall that the \emph{intervally disjoint union} $G\sqcup G'$ is a vertex disjoint union of $G$ and $G'$ where all vertices of $G$ are to left of all vertices of $G'$.
 Further the \emph{concatenation} $G\con G'$ of two ordered graphs $G$ and $G'$ is obtained from $G\sqcup G'$ by identifying the rightmost vertex in the copy of $G$ with the leftmost vertex in the copy of $G'$.
 For an integer $b>0$ we shall write $\sqcup_b G$ and $\con_b G$ for an intervally disjoint union respectively a concatenation of $b$ copies of $G$.
 Finally $\hang{a}{b}{G}$ denotes the ordered graph obtained from $\vec{S}_a\sqcup(\sqcup_b G)$ by connecting the leftmost vertex of this union with the leftmost vertex of each of the $b$ copies of $G$.
 See Figure~\ref{fig:toolbox} for an illustration.
 \begin{figure}
  \centering
  \includegraphics{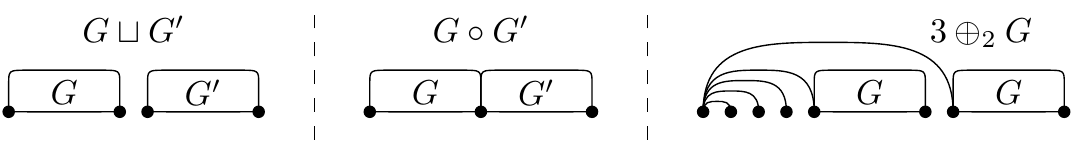}
  \caption{Three different types of unions of ordered graphs.}
  \label{fig:toolbox}
 \end{figure}

 Given integers $i$, $j$ with $1\leq j\leq i$, and a sequence $d=d_1,d_2,\ldots$ of positive integers, let $H_i(d)=\vec{S}_{d_i}\con\cdots\con \vec{S}_{d_1}$ and $H_i^j(d)=\vec{S}_{d_i}\con\cdots\con \vec{S}_{d_j}$ be right caterpillars.
 For convenience let $H_0(d)$ and $H_i^{i+1}(d)$ each denote a single vertex ordered graph.
 Consider a right star $H$ and a sequence $d=d_1,d_2,\ldots$ of positive integers.
 A \emph{left determiner for $(H,H_i(d))$}, $i\geq 0$, is an ordered graph $F$ such that
 \begin{itemize}
  \item for any $2$-coloring of the edges of $F$ without red copies of $H$ there is a blue copy of $H_i(d)$ that contains the leftmost vertex of $F$, and
  
  \item there is a \emph{good} coloring of the edges of $F$, that is, a $2$-coloring without red copies of $H$ or blue copies of $H_{i+1}(d)$ such that there is a unique blue copy of $H_i(d)$ that contains the leftmost vertex of $F$ and this copy is induced and isolated in the blue subgraph.
 \end{itemize}
 A \emph{right determiner for $(H,H_i^j(d))$}, $1\leq j\leq i+1$, is an ordered graph $F$ such that
 \begin{itemize}
  \item for any $2$-coloring of the edges of $F$ without red copies of $H$ or blue copies of $H_i(d)$ there is a blue copy of $H_i^j(d)$ that contains the rightmost vertex of $F$, and
  
  \item there is a \emph{good} coloring of the edges of $F$, that is, a $2$-coloring without red copies of $H$ or blue copies of $H_i(d)$ such that there is a unique blue copy of $H_i^j(d)$ that contains the rightmost vertex of $F$ and this copy is induced and isolated in the blue subgraph.
 \end{itemize}
 See Figure~\ref{fig:determiners} for examples of determiners.
 \begin{figure}
  \centering
  \includegraphics{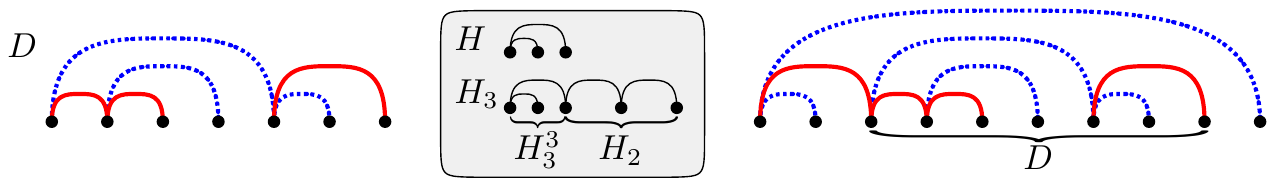}
  \caption[Left and right determiners with respective good colorings.]{A left determiner for $(H,H_2(d))$ (left) and a right determiner for $(H,H_3^3(d))$ (right) with respective good colorings (red solid  and blue dashed edges). Here $H$ is a right star on two edges and $d_1=d_2=1$ and $d_3=2$.}
  \label{fig:determiners}
 \end{figure}
 \begin{lemma}\label{lem:determiners}
  Let $H$ be a  right star with at least one edge, let $d$ be a sequence of positive integers, and let $i$, $j$ be non-negative integers with $j\leq i+1$.
  Then there is a left determiner for $(H,H_i(d))$ and, if $j\geq 2$, there is a right determiner for $(H,H_i^j(d))$.
 \end{lemma}
 \begin{proof}
  Let $s=\lvert E(H)\rvert $, $d=d_1,d_2,\ldots$, $H_i=H_i(d)$, and $H_i^j=H_i^j(d)$.
  First we shall construct a left determiner for $(H,H_i)$ by induction on $i$.
  It is easy to see that a single vertex graph is left determiner for $(H,H_0)$ and a right star on $s+d_1-1$ edges is a left determiner for $(H,H_1)$.
  Suppose that $i\geq 2$ and let $D_{\leq i-1}$ denote a left determiner for $(H,H_{i-1})$, which exists by induction.
  Let $D=\hang{(d_i-1)}{s}{D_{\leq i-1}}$, let $D_1,\ldots,D_s$ denote the copies of $D_{\leq i-1}$ in $D$, and let $u_t$ denote the leftmost vertex of $D_t$, $1\leq t\leq s$, see Figure~\ref{fig:leftDeterminer} for an illustration.
  
%
 \begin{figure}
  \centering
  \includegraphics{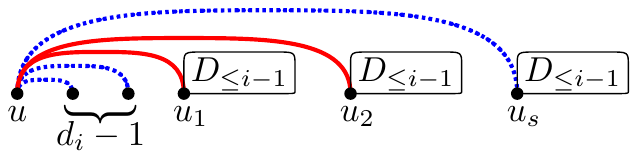}
  \caption[A left determiner with a good coloring.]{A left determiner for $(H,H_i)$ with $s=\lvert E(H)\rvert =3$, $d_i=3$, and a good coloring of its edges. 
   Here $D_{\leq i-1}$ is a left determiner for $(H,H_{i-1})$.}
  \label{fig:leftDeterminer}
 \end{figure}

  To see that $D$ is a left determiner for $(H,H_i)$ consider a $2$-coloring of its edges without a red copy of $H$.
  Let $u$ denote the leftmost vertex of $D$.
  Since $u$ is of degree $d_i+s-1$, it is incident to at least $d_i$ blue edges.
  Consider the rightmost vertex $v$ such that the edge $uv$ is blue.
  Then $v=u_t$ for some $t\in[s]$, and hence $v$ is leftmost in a blue copy of $H_{i-1}$.
  Thus there is a blue copy of $H_{i}$ that contains $u$.
  
  It remains to give a good coloring of the edges of $D$.
  Recall that $D_{\leq i-1}$ has a good coloring, that is, a $2$-coloring of its edges without red copies of $H$ or blue copies of $H_i$ such that the blue copy of $H_{i-1}$ that contains the leftmost vertex of $D_{\leq i-1}$ is induced and isolated in the blue subgraph.
  Color $D_1,\ldots, D_s$ according to such a coloring.
  Moreover color some $s-1$ of the edges $uu_t$, $1\leq t\leq s$, in red, and all other edges incident to $u$ in blue.
  See Figures~\ref{fig:determiners} and~\ref{fig:leftDeterminer}.
  Clearly there is no red copy of $H$ and no blue copy of $H_{i+1}$.
  Moreover the blue copy of $H_{i}$ that contains $u$ is  induced and isolated in the blue subgraph.
  Hence this coloring is a good coloring of $D$.
  Thus $D$ is a left determiner for $(H,H_i)$.
  
  \medskip
  
  Next we shall construct right determiners for $(H,H_i^j)$ for $i$, $j$ with $i+1\geq j\geq 2$.
  Consider some fixed $i\geq 1$.
  We shall construct a right determiner for $(H,H_i^j)$ by induction on $i+1-j$ (that is, ``from left to right''), using the already constructed left determiners.
  We see that a single vertex graph is a right determiner for $(H,H_i^{i+1})$.
  This forms the base case $i+1-j=0$, that is, $j=i+1$.
  Suppose that $i+1-j> 0$, that is, $j\leq i$.
  Let $D_{\geq j+1}$ denote a right determiner for $(H,H_i^{j+1})$, which exists by induction, and let $D_{\leq j-1}$ denote a left determiner for $(H,H_{j-1})$ (note that $j-1\geq 1$).
  Let $D'=\hang{(d_j-1)}{s-1}{D_{\leq j-1}}$, let $x$ denote the leftmost vertex in $D'$, let $D_1,\ldots,D_{s-1}$ denote the copies of $D_{\leq j-1}$ in $D'$, and let $v_t$ denote the leftmost vertex in $D_t$, $1\leq t\leq s-1$.
  Obtain an ordered graph $D$ from $D_{\geq j+1}\con D'$ by adding a vertex $y$ to the right of all other vertices and an edge between $x$ and $y$.
  See Figure~\ref{fig:rightDeterminer}.
 \begin{figure}
  \centering
  \includegraphics{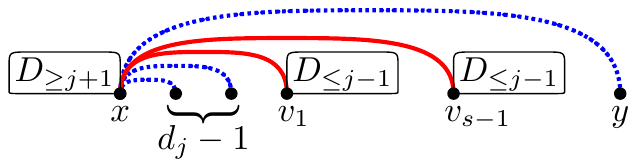}
  \caption[A right determiner with a good coloring.]{A right determiner for $(H,H_i^j)$ with $s=\lvert E(H)\rvert =3$, $d_j=3$, and a good coloring of its edges.
   Here $D_{\leq j-1}$ is a left determiner for $(H,H_{j-1})$ and $D_{\geq j+1}$ is a right determiner for $(H,H_i^{j+1})$.}
  \label{fig:rightDeterminer}
 \end{figure}

  We claim that $D$ is a right determiner for $(H,H_i^j)$.
  Consider a $2$-coloring of the edges of $D$ without red copies of $H$ or blue copies of $H_i$.
  We shall find a blue copy of $H_i^j$ that contains $y$.
  By construction $x$ is rightmost in a blue copy of $H_i^{j+1}$.
  Moreover $x$ is left endpoint of at least $d_j$ blue edges.
  Assume that the edge $xy$ is red.
  Consider the rightmost vertex $z$ such that the edge $xz$ is blue.
  Then $z=v_t$ for some $t$, $1\leq t\leq s-1$, and hence $z$ is leftmost in a blue copy of $H_{j-1}$.
  Thus there is a blue copy of $H_i$, a contradiction.
  This shows that $xy$ is colored blue and there is a blue copy of $H_i^j$ that contains $y$.
  
  It remains to give a good coloring of $D$.
  Recall that $D_{\geq j+1}$ has a good coloring, that is, a $2$-coloring of its edges without red copies of $H$ or blue copies of $H_i$ such that the blue copy of $H_i^{j+1}$ that contains $x$ is induced and isolated in the blue subgraph.
  Similarly $D_{\leq j-1}$ has a good coloring.
  Color the leftmost copy of $D_{\geq j+1}$ in $D$ and each $D_t$, $1\leq t\leq s-1$, according to such colorings.
  Color the edges $xv_t$ red, $1\leq t\leq s-1$, and all remaining edges with left endpoint $x$ blue.
  See Figures~\ref{fig:determiners} and~\ref{fig:rightDeterminer} for an illustration.
  Clearly there is no red copy of $H$ and no blue copy of $H_i$ (since $j\geq 2$).
  Moreover the blue copy of $H_i^j$ that contains $y$ is  induced and isolated in the blue subgraph.
  Hence $D$ is a right determiner for $(H,H_i^j)$.
 \end{proof}
 
 Finally we need the following structural observation on $\chi$-unavoidable ordered graphs.
 A \emph{bonnet} is an ordered graph on four or five vertices $u_1< u_2\leq u_3< u_4\leq u_5$ with edge set $\{u_1u_2, u_1 u_5, u_3u_4\}$, or on vertices $u_1 \leq u_2< u_3 \leq u_4 < u_5$ with edge set $\{u_1u_5, u_4u_5, u_2u_3\}$.
 Two edges $xy$, $x<y$, and $x'y'$, $x'<y'$, are \emph{crossing} is $x<x'<y<y'$ or $x'<x<y'<y$.
 An ordered path $P=u_1\cdots u_n$ is \emph{tangled} if for a vertex $u_i$, with $1<i<n$, that is either leftmost or rightmost in $P$ there is an edge in the subpath $u_1\cdots u_i$ that crosses an edge in the subpath $u_i\cdots u_n$.
 
 \begin{lemma}\label{lem:leftDeg1}
 Let $G$ be a $\chi$-unavoidable connected ordered graph with at least one edge where each vertex has at most one neighbor to the left (right).
 Then $G$ is a right (left) caterpillar.
 \end{lemma}
 \begin{proof}
  Suppose that each vertex in $G$ has at most one neighbor to the left.
  We shall show that $G$ is a right caterpillar.
  Since $G$ is $\chi$-unavoidable and connected it is a tree and contains neither a bonnet nor a tangled path~\cite{ForbiddenOrderedSubgraphs}.
  Call a vertex $a$ of $G$ \emph{displayed} if there is no edge $bc$ in $G$ with $b<a<c$ (such vertices are called inner cut vertices in~\cite{ForbiddenOrderedSubgraphs}).
  Note that a segment of a right caterpillar contains exactly two displayed vertices, namely its leftmost and rightmost vertex.
  We call a subgraph of $G$ a \emph{potential segment} if it is induced by an interval $I$ in $G$ whose leftmost and rightmost vertex are displayed in $G$ and all other vertices in $I$ are not displayed in $G$.
  Let $G'$ be a potential segment.
  We shall show that $G'$ is a right star.
  Note that $G'$ is connected, since $G$ is connected, and a vertex in $G'$ is displayed in $G$ if and only if it is displayed in $G'$.
  Let $u$ denote the leftmost vertex of $G'$ and let $v$ be a neighbor of $u$ in $G'$.
  For the sake of a contradiction assume that there is an edge $vv'$ in $G'$ with $v<v'$.
  By definition of $G'$ the vertex $v$ is not displayed in $G$, and thus not in $G'$.
  Hence there is an edge $xy$ in $G'$ with $u\leq x<v<y$.
  We have $y\neq v'$ since $v'$ has only one neighbor to the left.
  If $u=x$ then $\{u, v, v', y\}$ forms a tangled path in case $y<v'$, and a bonnet in case $v'<y$.
  If $u\neq x$ then consider a path in $G'$ that contains $uv$ and $xy$, which exists since $G'$ is connected.
  Since $uv$ and $xy$ are crossing and each vertex in $G'$ has at most one neighbor to the left we see that this path is tangled.
  In either case there is a contradiction and hence no edge $vv'$ with $v<v'$ exists in $G'$.
  If $v'v$ is an edge in $G'$ with $v'<v$ then $u=v'$ by assumption.
  Altogether $G'$ is a right star.
  This shows that every potential segment of $G$ is a right star.
  Therefore $G$ is a right caterpillar (as $G$ contains at least on edge).
  If each vertex in $G$ has at most one neighbor to the right, then  $G$ is a left caterpillar due to symmetric arguments.
 \end{proof}

 
\begin{proof}[Proof of Theorem~\ref{thm:UnavoidInfinite}] 
 Recall that $(H,H')$ is a Ramsey finite pair of $\chi$-unavoidable connected ordered graphs with at least two edges each.
 First we shall prove that $(H,H')$ is a pair of a right star and a right caterpillar or a pair of a left star and a left caterpillar.
 Then, using the determiners introduced above, we give constructions of infinitely many minimal ordered Ramsey graphs of such pairs if the caterpillar is not almost increasing.

 Since $H$ and $H'$ are connected and $\chi$-unavoidable both $H$ and $H'$ are trees.
 Since $(H,H')$ is Ramsey finite there is a forest in $\ors(H,H')$ due to Theorem~\ref{thm:unavoidNoRamseyForest}.
 Due to Theorem~\ref{thm:OrderedRamseyForest} and since $H$ and $H'$ are connected and have at least two edges each, we assume, without loss of generality, that $H$ is a right star while each vertex of $H'$ has at most one neighbor to the left.
 Since $H'$ is a $\chi$-unavoidable tree, $H'$ is a right caterpillar  due to Lemma~\ref{lem:leftDeg1}.
 Let $d=d_1,\ldots,d_i$ denote the defining sequence of $H'=H_i(d)$.
 Recall that $H'$ is almost increasing if $d_2\leq\cdots\leq d_i$ and, if $i\geq 3$, $d_1\leq d_3$.
 In particular if $H'$ is not almost increasing, then either there is some $j$, $1\leq j\leq i-2$, with $d_j>\max\set{d_{j+1},d_{j+2}}$, or there is some $j$, $i\geq j\geq 3$, with $d_{j-1} > d_{j}$.
 Below we give constructions of infinitely many minimal ordered Ramsey graphs for both cases, showing that $H'$ is indeed almost increasing.
 
 Recall that $H_t=H_t(d)$ and $H_i^{t+1}=H_i^{t+1}(d)$ are the subgraphs of $H_i(d)$ that consist of the $t$ rightmost segments respectively the $i-t$ leftmost segments of $H_i(d)$, $0\leq t\leq i+1$.
 Let $D_{\leq t}$ be a left determiner for $(H,H_t)$, $0\leq t<i$, and let $D_{\geq t}$ be a right determiner for $(H,H_i^t)$, $2\leq t\leq i+1$, which exist due to Lemma~\ref{lem:determiners}.
 
 \begin{proofEnum}[label=\textit{Case \arabic{proofEnumi}}:]
 \item There is $j$, $1\leq j\leq i-2$, with $d_j>\max\set{d_{j+1},d_{j+2}}$.
 We shall prove that $(H,H_i(d))$ is Ramsey infinite by constructing infinitely many minimal Ramsey graphs.  
 Let $a=\max\{d_{j+2},d_{j+1}\}-1$.
 Obtain a graph $\Gamma'$ from $(\hang{a}{\lvert E(H)\rvert -1}{D_{\leq j}})\sqcup D_{\geq j+3}$ by adding an edge between the leftmost and the rightmost vertex.
 Similarly obtain a graph $\Gamma''$ from $(\hang{a}{\lvert E(H)\rvert -1}{D_{\leq j+1}})\sqcup D_{\geq j+3}$ by adding an edge between the leftmost and the rightmost vertex.
 For $n\geq 1$ let $\Gamma_n= D_{\geq j+3}\con\Gamma''\con (\con_n \Gamma')\con D_{\leq i}$.
 See Figure~\ref{fig:unvoidInf1} for an illustration in case $\lvert E(H)\rvert =2$.
\begin{figure}
  \centering
  \includegraphics{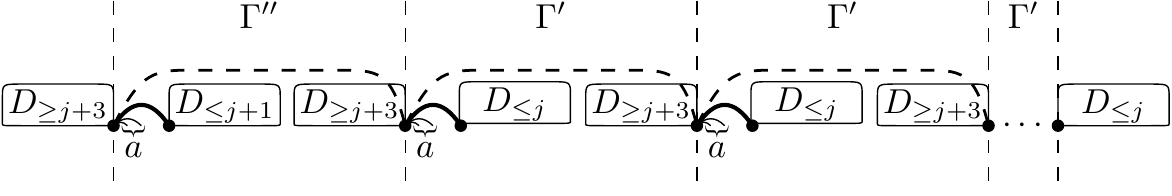}
  \caption[A Ramsey graph for $(H,H_i)$ where $H$ is a right star on two edges and $H_i$ is a right caterpillar.]{A Ramsey graph for $(H,H_i)$ where $H$ is a right star on two edges and $H_i$ is a right caterpillar with at least three segments where $a=\max\{d_{j+2},d_{j+1}\}-1$ and $d_j>a+1$.}
  \label{fig:unvoidInf1}
 \end{figure}
%
 
 
 First we shall prove that $\Gamma_n\to(H,H_i)$.
 We refer to bold and dashed edges like given in Figure~\ref{fig:unvoidInf1}, that is, an edge is dashed if it is the longest edge in the copy $\Gamma''$ or in one of the copies of $\Gamma'$, and an edge is bold if it has the same left endpoint as some dashed edge and its right endpoint is leftmost in a copy of $D_{\leq j+1}$ (in $\Gamma''$) or $D_{\leq j}$ (in $\Gamma'$).
 Consider a $2$-coloring of the edges of $\Gamma_n$ without red copy of $H$.
 Observe that the left endpoint of each dashed edge is rightmost in a blue copy of $H_i^{j+3}$ and left endpoint of at least $a+1\geq d_{j+2}$ blue edges.
 Hence, if a dashed edge is blue, then its right endpoint is rightmost in a blue copy of $H_i^{j+2}$.
  
 First suppose that all dashed edges are blue.
 Consider the rightmost copy $K$ of $\Gamma'$ and its dashed edge $xy$, $x<y$.
 Then $x$ is rightmost in a blue copy of $H_i^{j+2}$ and $y$ is leftmost in a blue copy of $H_{j}$.
 We see that the blue edge $xy$ with $a$ further blue edges in $K$ yields a blue copy of $H_i$.
 
 Now suppose that the dashed edge $uv$, $u<v$, in $\Gamma''$ is red.
 Note that $u$ is rightmost in a blue copy of $H_i^{j+3}$.
 Consider the rightmost vertex $z$ such that the edge  $uz$ is blue.
 Since there are $a+1$ blue edges with left endpoint $u$ (and $uv$ is red), the edge $uz$ is a bold edge.
 Since $a\geq d_{j+2}-1$ and $z$ is the leftmost vertex in a blue copy of $H_{\leq j+1}$, there is a blue copy of $H_i$.
  
 Finally suppose that there is a blue dashed edge whose right endpoint $w$ is incident to a red dashed edge.
 Consider the rightmost vertex $z$ such that the edge  $wz$ is blue.
 Since there are $a+1$ blue edges with left endpoint $w$ (and the dashed edge with left endpoint $w$ is red), the edge $wz$ is a bold edge.
 Recall that $w$ is rightmost in a blue copy of $H_i^{j+2}$.
 Since $a\geq d_{j+1}-1$ and $z$ is leftmost in a blue copy of $H_{j}$, there is a blue copy of $H_i$.
 Altogether $\Gamma_n\to(H,H_i)$.
 
 
 Next we shall show that each minimal Ramsey graph of $(H,H_i)$ that is a subgraph of $\Gamma_n$ contains all dashed edges, that is, contains at least $n+1$ edges.
 Let $\bar{\Gamma}$ be obtained from $\Gamma_n$ by removing some dashed edge $\bar{e}$.
 We construct a coloring of $\bar{\Gamma}$ without red copies of $H$ or blue copies of $H_i$ as follows.
 Note that $\bar{\Gamma}$ consists of two connected components.
 First consider the component that contains the left endpoint of $\bar{e}$.
 Color all bold edges in this component in red, all other edges with left endpoint equal to the left endpoint of some bold edge in blue.
 The remaining edges form vertex disjoint copies of $D_{\geq j+3}$, $D_{\leq j+1}$, and $D_{\leq j}$.
 Color each of these determiners according to a good coloring.
 There is no red copy of $H$ since at most $\lvert E(H)\rvert$ bold edges have a common left endpoint.
 Moreover there is no blue copy of $H_i$ within one of the determiners.
 The blue edges not in one of the determiners form a right caterpillar where each segment has $a+1$ edges.
 Since $a+1<d_j$ and since the blue copies of $H_i^{j+3}$ in the copies of $D_{\geq j+3}$ are induced and isolated, there is no blue copy of $H_i$.
 
 Now consider the component that contains the right endpoint of $\bar{e}$.
 For each vertex $p$ in this component that is the left endpoint of some dashed edge color this dashed edge and $\lvert E(H)\rvert -2$ further edges with left endpoint $p$ in red and all other edges with left endpoint $p$ blue.
 The remaining edges form vertex disjoint copies of $D_{\geq j+3}$ and $D_{\leq j}$.
 Color each of these determiners according to a good coloring.
 Clearly there is no red copy of $H$.
 Each component of the blue subgraph is contained in a copy of $D_{\geq j+3}\con(\hang{a}{\lvert E(H)\rvert -1}{D_{\leq j}})$.
 Similar as before we see that there is no blue copy of $H_i$.
 This shows that $\bar{\Gamma}\not\in\ors(H,H_i)$.
 
 \medskip
 
 For each $n\geq 1$ choose a minimal Ramsey graph of $(H,H')$ contained in $\Gamma_n$.
 The arguments before show that infinitely many of these graphs are pairwise non-isomorphic.
 Thus $(H,H')$ is Ramsey infinite.


 \item There is $j$, $i\geq j\geq 3$, with $d_{j-1} > d_{j}$.
 We shall prove that $(H,H_i)$ is Ramsey infinite by constructing infinitely many minimal Ramsey graphs.
 An illustration of the following construction is given in Figure~\ref{fig:unvoidInf3}.
 Let $\Gamma$ denote an ordered graph obtained from $\vec{S}_{d_j-1}\sqcup D_{\geq j+1}$ by adding an edge between the leftmost and the rightmost vertex (recall that $D_{\geq j+1}$ is a right determiner for $(H,H_{i}^{j+1})$).
 For $n\geq 1$ let $F'_n$ be defined as follows.
 Start with a right determiner $D_{\geq j}$  for $(H,H_{i}^{j})$ and let $x<y$ denote its two rightmost vertices.
 Add a copy of $D_{\geq j+1}$ that has all its vertices to the right of $x$ and has $y$ as its rightmost vertex.
 Call the resulting graph $D$.
 To this graph $D$ concatenate $n$ copies $\Gamma_1,\ldots,\Gamma_n$ of $\Gamma$ and a left determiner $D_{\leq j-1}$ for $(H,H_{j-1})$, one after the other in this order.
 Finally add $d_{j-1}-d_j>0$ isolated vertices and an intervally disjoint union of $\lvert E(H)\rvert -1$ left determiners $D_{\leq j-2}$ for $(H,H_{j-2})$ to the right of all current vertices.
 Altogether
 \[F'_n=D\con(\con_n\Gamma)\con D_{\leq j-1}\sqcup(\sqcup_{d_{j-1}-d_j}K_1)\sqcup(\sqcup_{\lvert E(H)\rvert -1}D_{\leq j-2}).\]
 Let $U$ be the set of isolated vertices and let $W$ denote the set of leftmost vertices of the graphs $D_{\leq j-2}$ added in the last step.
 Let $\gamma_t$ denote the leftmost vertex of $\Gamma_t$ in $F$, $1\leq t\leq n$ and let $\gamma_{n+1}$ be the rightmost vertex of $\Gamma_n$.
 We obtain an ordered graph $F_n$ from $F'_n$ by adding a complete bipartite graph between $U\cup W$ and $\{\gamma_1,\ldots,\gamma_n\}$.
 \begin{figure}
  \centering
  \includegraphics{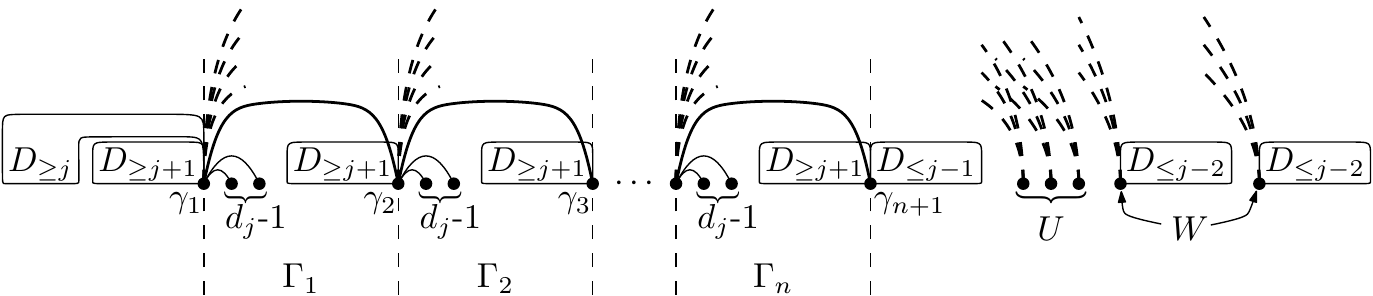}
  \caption[A Ramsey graph for $(H,H_i)$ where $H$ is a right star and $H_i$ is a right caterpillar whose defining sequence contains $d_j<d_{j-1}$ for some $j$, $i\geq j\geq 3$.]{A Ramsey graph for $(H,H_i)$ where $H$ is a right star and $H_i$ is a right caterpillar whose defining sequence contains $d_j<d_{j-1}$ for some $j$, $i\geq j\geq 3$. The dashed edges form a complete bipartite graph.}
  \label{fig:unvoidInf3}
 \end{figure}
 
 \medskip
 
 First we shall prove that $F_n\to(H,H')$.
 For the sake of contradiction consider a $2$-coloring of the edges of $F=F_n$ without red copies of $H$ or blue copies of $H'$. 
 We shall prove that $\gamma_t$ is rightmost in a blue copy of $H_i^j$, $1\leq t\leq n+1$, by induction on $t$.
 For $t=1$ this holds since $\gamma_1$ is rightmost in a right determiner for $H_i^j$.
 Consider $t>1$.
 Since $\gamma_{t-1}$ is left endpoint of $d_j+\lvert U\rvert+\lvert W\rvert=d_{j-1}+\lvert E(H)\rvert -1$ edges there are at least $d_{j-1}$ blue edges with left endpoint $\gamma_{t-1}$.
 Consider the rightmost vertex $z$ with $\gamma_{t-1}z$ colored blue.
 Since $\gamma_{t-1}$ is rightmost in a blue copy of $H_i^j$, $z$ is rightmost in a blue copy of $H_i^{j-1}$.
 Hence $z\not\in W$, since otherwise there is a blue copy of $H_i$ as each vertex in $W$ is leftmost in a blue copy of $H_{j-2}$. 
 Hence all edges between $W$ and $\gamma_{t-1}$ are red and, since $\lvert W\rvert =\lvert E(H)\rvert -1$, all other edges with left endpoint $\gamma_{t-1}$ are blue.
 In particular $\gamma_{t-1}\gamma_t$ and $d_{j}-1$ further edges $\gamma_{t-1}z$, with $\gamma_{t-1}<z<\gamma_{t}$, are blue.
 Since there is a blue copy of $H_i^{j+1}$ with rightmost vertex $\gamma_{t-1}$, $\gamma_t$ is rightmost is a blue copy of $H_i^j$. 
 These arguments show that $\gamma_{n+1}$ is rightmost in a blue copy of $H_i^j$.
 This forms a blue copy of $H_i$ together with a blue copy of $H_{j-1}$ in the left determiner $D_{\leq j-1}$ with leftmost vertex $\gamma_{n+1}$, a contradiction.
 Therefore $F_n\to(H,H')$.
 
 \medskip
 
 Next we shall show that each minimal Ramsey graph of $(H,H')$ that is a subgraph of $F_n$ contains all edges $\gamma_t\gamma_{t+1}$, $1\leq t\leq n$.
 Let $\bar{F}$ be obtained from $F_n$ by removing the edge $\gamma_s\gamma_{s+1}$ for some $s$, $1\leq s\leq n$.
 We construct a coloring of $\bar{\Gamma}$ without red copies of $H$ or blue copies of $H'=H_i$ as follows.
 For each $t\leq s$ color all edges between $\gamma_t$ and $W$ red and all other edges with left endpoint $\gamma_t$ blue.
 For each $t$, $s+1\leq t\leq n$, color the edge $\gamma_t\gamma_{t+1}$ red and all other edges with left endpoint $\gamma_t$ blue.
 The remaining edges are contained in an edge disjoint union of determiners and are colored according to a good coloring of each determiner.
 There are no red copies of $H$ since a good coloring of a determiner has no red copy of $H$ and each $\gamma_t$, $1\leq t\leq n$ is left endpoint of at most $\lvert W\rvert =\lvert E(H)\rvert -1$ red edges.
 Assume that there is a blue copy $\bar{H}$ of $H_i$.
 Consider the unique vertex $u$ in $\bar{H}$ that is contained in a (blue) copy of $H_j$ and a (blue) copy of $H_i^j$.
 Due to the good colorings of the determiners we have $u=\gamma_t$ for some $t\in[n+1]$.
 For each $t\geq s+1$ the vertex $\gamma_t$ is not leftmost in a blue copy of $H_{j}$ since $j\geq 3$.
 Hence there are no blue copies of $H_i$ containing a vertex $\gamma_t$ with $t\geq s+1$.
 Consider the vertices $\gamma_t$ for $t\leq s-1$.
 Each of these is left endpoint of $d_{j-1}$ blue edges, but $\gamma_t\gamma_{t+1}$ is the only such blue edge whose right endpoint has further neighbors to the right.
 Note that there are only $d_j-1$ neighbors $z$ of $\gamma_t$ with $\gamma_t<z<\gamma_{t+1}$.
 Since $d_j<d_{j-1}$ and $j\geq 3$ no vertex $\gamma_t$ with $t\leq s-1$ is leftmost in a blue copy of $H_{j-1}$.
 Therefore no such vertex $\gamma_t$ is leftmost in a blue copy of $H_{j}$, since $\gamma_{t+1}$ is leftmost in a blue copy of $H_{j-1}$ otherwise.
 This shows that there is no blue copy of $H_i$ and $\bar{F}\not\in\ors(H,H_i)$.
 
 \medskip
 
 For each $n\geq 1$ choose a minimal Ramsey graph of $(H,H')$ contained in $F_n$.
 The arguments before show that infinitely many of these graphs are pairwise non-isomorphic.
 Thus $(H,H')$ is Ramsey infinite.\qedhere
 \end{proofEnum}
 \end{proof}


 \subsection{Proof of Theorem~\ref{thm:CaterpillarFinite}}\label{sec:unavoidFinite}

 Recall that $(H,H')$ is a pair of a right star and a right caterpillar or a pair of a left star and a left caterpillar and $d_1,\ldots,d_i$ is the defining sequence of the caterpillar.
 Suppose that $i\leq 2$ or $d_1\leq\cdots\leq d_i$.
 In order to prove that $(H,H')$ is Ramsey finite we shall show that each minimal ordered Ramsey graph of $(H,H')$ is a member of a finite family of ordered graphs defined below.
 Without loss of generality assume that $H$ is a right star with $s$ edges and $H'$ is a right caterpillar.
Let $H_t$ denote the subgraph of $H'$ that consist of the $t$ rightmost segments of $H'$, $0\leq t\leq i$.
 Observe that $H_i=H'$. 
 Recursively define sets $\cF_j$, $1\leq j\leq i$, of ordered graphs as follows.
 Let $\cF_1=\{\vec{S}_{s+d_1-1}\}$ (recall that $\vec{S}_{p}$ is a right star on $p$ edges).
 Consider $j>1$.
 An ordered graph $F$ is in $\cF_j$ if and only if its leftmost vertex $u$ has exactly $s+d_j-1$ neighbors $v_1<\cdots<v_{s+d_j-1}$ and there are (not necessarily disjoint) subgraphs $F_1,\ldots,F_s$ of $F$ with $E(F-u)=\cup_{t=1}^sE(F_t)$, $F_t\in\cF_{j-1}$, and $v_{t+d_j-1}$ is leftmost in $F_t$, $1\leq t\leq s$.
 See Figure~\ref{fig:pointedUnion}.
 \begin{figure}
  \centering
  \includegraphics{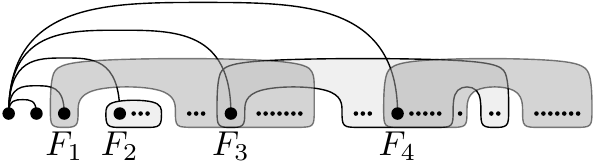}
  \caption[A Ramsey graph for $(H,H_i)$ where $H$ is a right star with four edges and $H_i$ is a right caterpillar with $d_i=2$.]{A Ramsey graph for $(H,H_i)$ where $H$ is a right star with four edges and $H_i$ is a right caterpillar with $d_i=2$. Here $F_{1}$, $F_2$, $F_3$,  $F_4\in\ors(H,H_{i-1})$ such that for any coloring without red copies of $H$ there is a blue copy of $H_{i-1}$ that contains the leftmost vertex of $F_t$, $t=1,2,3,4$. The graphs $F_1,\ldots,F_4$ might share vertices and edges as long as their leftmost vertices are mutually distinct.}
  \label{fig:pointedUnion}
 \end{figure}
%
 Note that for each $j\in[i]$ the set $\cF_j$ is finite.
 We shall show that each minimal ordered Ramsey graph of $(H,H_i)$ is in $\cF_i$.
 Hence $(H,H_i)$ is Ramsey finite.

 First of all observe that for each $j\in[i]$ each graph in $\cF_j$ is in $\ors(H,H_j)$.
 Even more, for each coloring of the edges of some graph $F\in\cF_j$ without red copies of $H$ there is a blue copy of $H_j$ containing the leftmost vertex of $F$.

 \medskip

 If $i=1$, then $H'=H_1$.
 It is easy to see that $F\in\ors(H,H_1)$ if and only if $F$ contains a copy of $\vec{S}_{s+d_1-1}$.
 Therefore $\vec{S}_{s+d_1-1}$ is the only minimal ordered Ramsey graph of $(H,H_1)$.
 In particular each graph in $\ors(H,H_1)$ contains some member of $\cF_1=\{\vec{S}_{s+d_1-1}\}$ and $(H,H_1)$ is Ramsey finite.
 Now consider the case $i=2$ and some ordered graph $F$ that does not contain copies of any member of $\cF_2$.
 We shall give a coloring of the edges of $F$ without red copies of $H$ or blue copies of $H_2$.
 Let $D_1$ denote the set of all vertices in $F$ that are leftmost in a copy of some member of $\cF_1=\{\vec{S}_{s+d_1-1}\}$ in $F$.
 For $u\in V(F)$ let $r(u)$ denote its \emph{right degree}, that is, the number of edges $uv$ in $F$ with $u<v$.
 Note that $u\in D_1$ if and only if $r(u)\geq s+d_1-1$.
 We color the edges of $F$ in three steps.
 In the first step color each edge $uv$, with $u<v$, red if $v\in D_1$, and there are $d_2-1$ vertices $z$ with $u< z< v$. 
 In the second step color arbitrary further edges red, such that for each $u\in V(F)$ there are in total exactly $\min\{s-1,r(u)\}$ red edges $uv$ with $u<v$.
 In the last step color all yet uncolored edges blue.
 First assume for the sake of a contradiction that there is a blue copy of $H_2$.
 Let $uv$ denote the longest edge incident to the leftmost vertex $u$ in this copy.
 Then $v$ is leftmost in a blue copy of $H_1$ and hence $r(v)\geq s+d_1-1$ due to the second step.
 In particular $v\in D_1$.
 Moreover there are $d_2-1$ vertices $z$ with $u<z<v$.
 Hence $uv$ is colored red in the first step, a contradiction as $uv$ is blue.
 Next assume that there is a red copy of $H$.
 Then its was created in the first step.
 Hence the leftmost vertex $u$ of this red copy of $H$ has $s+d_2-1$ neighbors to the right in $F$, the $s$ rightmost of which are contained in $D_1$.
 Thus $u$ is leftmost in a copy of some graph from $\cF_2$, a contradiction.
 Altogether $F\not\in\ors(H,H_2)$.
 This proves that a graph is in $\ors(H,H_2)$ if and only if it contains a copy of some $F'\in\cF_2$.
 Since each member of $\cF_2$ is in $\ors(H,H_2)$ each minimal ordered Ramsey graph of $(H,H_2)$ is in $\cF_2$.
 Therefore $(H,H_2)$ is Ramsey finite.
 
 \medskip

 Finally consider the case $i\geq 3$ and an ordered graph $F$ that does not contain copies of any member of $\cF_i$.
 We shall give a coloring of the edges of $F$ without red copies of $H$ or blue copies of $H_i$.
 By assumption we have $d_1\leq\cdots\leq d_i$. 
 Observe that there is a copy of $H_{j-1}$ in $H_{j}$ that contains the leftmost vertex of $H_{j}$ for each $j$, $2\leq j\leq i$.
 Moreover, the leftmost vertex of each $F\in\cF_j$ is contained in a copy of some $F'\in\cF_{j-1}$ in $F$, $2\leq j\leq i$.
 Recall that for each coloring of the edges of some graph $F\in\cF_j$ without red copies of $H$ there is a blue copy of $H_j$ containing the leftmost vertex of $F$.
 Hence, for each $t\in[j]$, there is also a blue copy of $H_t$ which contains the leftmost vertex of $F$ under such a coloring.
 Let $D_0=V(F)$ and for $j\in[i]$ let $D_j$ denote the set of all vertices in $F$ that are leftmost in a copy of some graph from $\cF_j$ in $F$.
 As argued above we have $\emptyset=D_{i}\subseteq D_{i-1}\cdots\subseteq D_1\subseteq D_0$.
 For $u\in V(F)$ let $h(u)$ denote the largest $j$ with $u\in D_j$.
 Color an edge $uv$, with $u<v$, red if and only if $h(u)\leq h(v)$ and there are $d_{h(u)+1}-1$ vertices $z$ with $u< z< v$.
 
 For the sake of a contradiction assume that there is a red copy $\bar{H}$ of $H$.
 Let $u$ denote the leftmost vertex in $\bar{H}$ and let $j=h(u)$.
 For each other vertex $v$ in $\bar{H}$ there are $d_{j+1}-1$ vertices $z$ with $u< z< v$, $h(v)\geq j$, and hence $v\in D_j$, as argued above.
 Thus $u$ is leftmost in a copy of some graph from $\cF_{j+1}$ in $F$, a contradiction as $h(u)=j$.
 See Figure~\ref{fig:canonicalColoring} (left).
\begin{figure}
 \centering
 \includegraphics{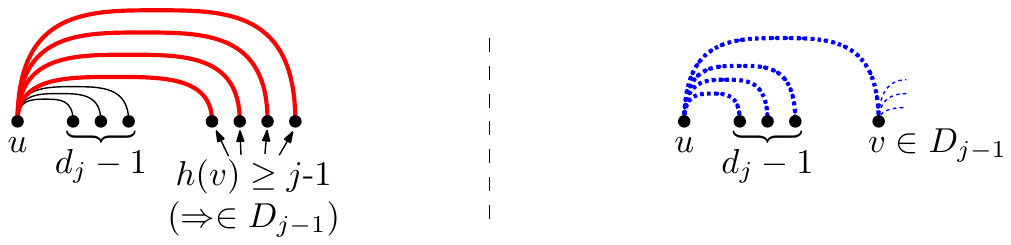}
 \caption[Illustration of coloring in proof of Theorem~\ref{thm:CaterpillarFinite}.]{All edges $uv$ with $u<v$, $h(u)=j-1$, and $h(v)\geq j-1$ are colored red if there are $d_j-1$ vertices between $u$ and $v$ (left). If $u\not\in D_j$, then $u$ is not leftmost in a red copy of $H$. Moreover if $u$ is leftmost in a blue copy of $H_j$, then $h(u)\geq j$ (and $u$ is in $D_j$) since otherwise $uv$ is colored red (right).}
 \label{fig:canonicalColoring}
\end{figure}

 Let $H_0$ denote the single vertex ordered graph.
 Next we shall prove by induction on $j$, $0\leq j\leq i$, that for each vertex $u$ which is leftmost in a blue copy of $H_j$ we have $u\in D_j$.
 This clearly holds for $j=0$.
 So consider $j>0$ and a blue copy $H''$ of $H_j$.
 Let $uv$ denote the longest edge incident to the leftmost vertex $u$ of $H''$.
 Then $v$ is leftmost in a blue copy of $H_{j-1}$ and hence $v\in D_{j-1}$ by induction.
 In particular $h(v)\geq j-1$.
 Moreover there are $d_j-1$ vertices $z$ with $u<z<v$.
 Hence $h(u)\geq j$, since otherwise $h(u)\leq h(v)$ and $d_j-1\geq d_{h(u)+1}-1$, and thus $uv$ is colored red.
 See Figure~\ref{fig:canonicalColoring} (right).
 Therefore $u\in D_j$.
 Since $D_i=\emptyset$ there is no blue copy of $H_i$.
 
 Altogether $F\not\in\ors(H,H_i)$.
 This proves that a graph is in $\ors(H,H_i)$ if and only if it contains a copy of some $F'\in\cF_i$ (since each member of $\cF_i$ is in $\ors(H,H_i)$, as argued above).
 Therefore, each minimal ordered Ramsey graph of $(H,H_i)$ is contained in $\cF_i$.
 In particular $(H,H_i)=(H,H')$ is Ramsey finite.\qed


\section{Conclusions}\label{sec:conclusion}

We study the structure of the set $\ors(H,H')$ of ordered Ramsey graphs for pairs $(H,H')$ of ordered graphs.
First of all we characterize all such pairs $(H,H')$ that have some forest in $\ors(H,H')$ in  Theorem~\ref{thm:OrderedRamseyForest}.
A pair of unordered forests has a forest as a Ramsey graph if and only if one of the forests is a star forest.
In contrast to this, we give pairs of ordered star forests that do not have any forest as a Ramsey graph.
Next Theorem~\ref{thm:OrderedRamseyPseudoforest} characterizes all pairs of connected ordered graphs that have a pseudoforest as a Ramsey graph.
Again it turns out that the pairs of ordered graphs that have pseudoforests as Ramsey graphs are much more restricted than in the unordered case.
We do not have a full answer for disconnected ordered graphs.
It might be true that Lemma~\ref{lem:noRamseyPseudoforest} covers all pairs of ordered graphs that do not have any pseudoforest as a Ramsey graph.
\begin{question}
 Which pairs of (disconnected) ordered graphs have a pseudoforest as a Ramsey graph?
\end{question}
Then we consider the question for which pairs of ordered graphs the set $\ors(H,H')$ contains only finitely many minimal elements.
The corresponding question in the unordered setting is answered whenever $H=H'$, but a complete answer in the asymmetric case is known only if one of $H$ or $H'$ is a forest (see Theorems~\ref{thm:UnorderedCycleInf}, \ref{thm:unorderedForestCycle}, \ref{thm:Faudree}).
Similar to the unordered setting we show that any ordered graph $H$ that contains a cycle is Ramsey infinite.
Moreover Corollary~\ref{cor:connectedFinite} shows that a connected ordered graph $H$ is Ramsey finite if and only if $H$ is a star with its center to the right or to the left of all its leaves (called left respectively right star).
This is in contrast to the unordered setting where a connected graph is Ramsey finite if and only if it is a star with an odd number of edges (see Theorems~\ref{thm:UnorderedCycleInf} and~\ref{thm:Faudree}).

With Theorem~\ref{thm:ordDensityRamseyInfinite} we establish a relation between the question for smallest densities of ordered Ramsey graphs and the question for Ramsey finiteness.
Using a result of R\"odl and Ruci\'nski~\cite{RR93} (Theorem~\ref{thm:avgDegRamsey}) we see that any Ramsey finite ordered graph has a pseudoforest as a Ramsey graph.
Now our results from the first part show that every Ramsey finite ordered graph is a star forest with strong restrictions on the centers of the stars, see Theorem~\ref{thm:randomApplied}.
Further we show that every Ramsey finite ordered graph which is $\chi$-unavoidable has a forest as a Ramsey graph by Theorem~\ref{thm:unavoidNoRamseyForest}.
This yields that any Ramsey finite $\chi$-unavoidable ordered graph is a forest of left stars or a forest of right stars.
We think that the assumption of $\chi$-unavoidable is not necessary here, see Conjecture~\ref{conj:OrderedFiniteForest} below.

\paragraph{Disconnected ordered graphs.}

To some extent vertex disjoint unions of (unordered) graphs are rather easy to handle with respect to their Ramsey graphs.
A Ramsey graph for a vertex disjoint union of graphs $G$ and $H$ is given by a vertex disjoint union of a Ramsey graph of $G$, a Ramsey graph of $H$, and a Ramsey graph of the pair $(G,H)$.
For ordered graphs there are many different vertex disjoint unions and we do not see a uniform way to build ordered Ramsey graphs.
This is one reason why beyond Theorems~\ref{thm:randomApplied} and~\ref{thm:unavoidNoRamseyForest} we do not have many results in the disconnected case, although we characterize all connected Ramsey finite graphs as mentioned above.
Some examples of ordered graphs which are not known to be Ramsey finite or infinite are given in Figure~\ref{fig:openOrderedRamsey}.
\begin{figure}
 \centering
 \includegraphics{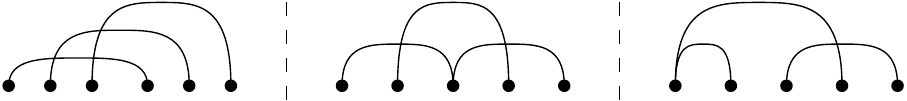}
 \caption{Ordered graphs that are not known to be Ramsey finite or infinite.}
 \label{fig:openOrderedRamsey}
\end{figure}

Most striking is our lack of understanding Ramsey graphs of ordered graphs with isolated vertices.
Suppose that $H$ is an ordered graph and $H'$ is obtained from $H$ by adding an isolated vertex.
If $F\in\ors(H)$ is minimal then a minimal graph in $\ors(H')$ is obtained by adding a suitable set of isolated vertices to $F$.
Therefore $H'$ is Ramsey infinite if $H$ is Ramsey infinite, but the reverse statement is open.
\begin{question}
 Is there a Ramsey finite ordered graph $H$ such that adding some isolated vertices to $H$ yields a Ramsey infinite ordered graph?
\end{question}
We consider the special case of intervally disjoint unions of Ramsey finite graphs in Theorem~\ref{thm:intervalUnionFinite} (which also does not cover isolated vertices).
While such a union turns out to be Ramsey finite the reverse statement is open.
\begin{question}
 Let $H$, $H'$ and $H''$ denote ordered graphs such that $(H,H'\sqcup H'')$ is Ramsey finite.
 Are both pairs $(H,H')$ and $(H,H'')$ Ramsey finite?
\end{question}
It is worth noting that the answer to the corresponding question in the unordered setting is negative.
For example let $H$ be the graph formed by a vertex disjoint union of a star on $5$ edges and a star on $2$ edges and let $H'$ denote a star on $3$ edges.
Then $(H,H')$ is Ramsey infinite~\cite{B81} and there is an integer $k$ such that adding $k$ isolated edges to $H'$ yields a Ramsey finite pair of graphs~\cite{Fa91} (see also Theorem~\ref{thm:Faudree}).

The main corollary of Theorem~\ref{thm:intervalUnionFinite} states that each pair of a monotone matching and any other ordered graph is Ramsey finite.
This is similar to Theorem~\ref{thm:unorderedForestCycle}\ref{enum:unordredMatchingFinite} stating that any pair of an (unordered) graph and some matching is Ramsey finite.
Moreover Theorem~\ref{thm:unorderedForestCycle}\ref{enum:unordredNotMatchingInfinite} states a result due to {\L}uczak~\cite{Lu94} that $(H,H')$ is Ramsey infinite for each graph $H$ which is not a matching and any other graph $H'$ which contains a cycle.
We think that also this property carries over to monotone matchings.
It might be possible to transfer the arguments from~\cite{Lu94} to ordered graphs to give a positive answer to the following question.
\begin{question}
 Let $H$ be an ordered graph that contains a cycle and let $H'$ be an ordered forest that is not a monotone matching.
 Is $(H,H')$ Ramsey infinite?
\end{question}

\paragraph{The asymmetric case.}
We think that Theorem~\ref{thm:ordDensityRamseyInfinite} generalizes to the asymmetric case as follows.
Kohayakawa and Kreuter~\cite{KK97} introduce the asymmetric $2$-density for a pair of graphs $(H,H')$ with $m_2(H)\geq m_2(H')$ given by $m_2(H,H')=\max\left\{\frac{\lvert E(H'')\rvert }{\lvert V(H'')\rvert -2+1/m_2(H)}\mid H''\subseteq H', \lvert E(H'')\rvert \geq 1\right\}$.
\begin{conjecture}\label{conj:asymmetricDensityRamseyInf}
 Let $H$ and $H'$ be an ordered graphs with $m_2(H)\geq m_2(H')$.
 If $m(F)>m_2(H,H')$ for each $F\in\ors(H,H')$, then $(H,H')$ is Ramsey infinite.
\end{conjecture}
One can see from the proof of Theorem~\ref{thm:ordDensityRamseyInfinite} that an asymmetric version of Theorem~\ref{thm:ordRandomRamsey} (on Ramsey properties of random graphs) is sufficient to prove Conjecture~\ref{conj:asymmetricDensityRamseyInf}.
Recently Gugelmann~\textit{et al.}~\cite{Gugel17} prove a slightly weaker statement for a pair $(H,H')$ of unordered graphs.
In the conclusions they claim that there is some $c$ such that for $p\geq c n^{-1/m_2(H,H')}\log(n)$ the probability that $G(n,p)$ is a Ramsey graph of $(H,H')$ tends to $1$ as $n$ tends to infinity.
Again one can  see from the proof of Theorem~\ref{thm:ordDensityRamseyInfinite} that such a bound on $p$ (in the ordered setting) is sufficient for a proof of Conjecture~\ref{conj:asymmetricDensityRamseyInf}.

Along with Conjecture~\ref{conj:asymmetricDensityRamseyInf}, a generalization of Theorem~\ref{thm:avgDegRamsey} to the asymmetric case would reveal more Ramsey infinite pairs of (ordered) graphs.
We propose the following conjecture based on the fact that each pair $(H,H')$ of graphs is Ramsey infinite provided that neither $H$ nor $H'$ is a matching and exactly one of $H$ or $H'$ contains a cycle~\cite{Lu94} (see Theorem~\ref{thm:unorderedForestCycle}).
\begin{conjecture}\label{conj:asymmetricRamseyDensity}
 Let $H$ and $H'$ be graphs with $m_2(H)\geq m_2(H')$.
 If neither $H$ nor $H'$ is a matching and at least one of $H$ or $H'$ contains a cycle, then $m(F)>m_2(H,H')$ for each $F\in\rs(H,H')$.
\end{conjecture}
If Conjecture~\ref{conj:asymmetricRamseyDensity} holds, then the result from~\cite{Gugel17} mentioned above shows that all pairs of (unordered) graphs which contain a cycle are Ramsey infinite and all Ramsey finite (unordered) graphs are characterized in~\cite{Fa91} (see Theorem~\ref{thm:Faudree}), settling the main open question in the unordered setting.
A similar discussion is given in~\cite{RamseyAlgo}.

\medskip

In the ordered setting several cases remain open, even if Conjectures~\ref{conj:asymmetricDensityRamseyInf} and~\ref{conj:asymmetricRamseyDensity} hold.
First note that pairs of ordered graphs involving a matching are not covered by either of the conjectures.
We discuss such pairs in the paragraph on disconnected ordered graphs above.
Next observe that $m_2(H,H')\leq 1$ for every pair $(H,H')$ of ordered forests.
Hence Conjecture~\ref{conj:asymmetricDensityRamseyInf} covers all such pairs which do not have a pseudoforest as a Ramsey graph.
It remains to consider such pairs having a pseudoforest as a Ramsey graphs, which corresponds to the result of Theorem~\ref{thm:ordDensityRamseyInfinite} from the symmetric case.
Theorem~\ref{thm:unavoidNoRamseyForest} applies to the asymmetric case as well and shows that a Ramsey finite pair of $\chi$-unavoidable ordered forests has a forest as a Ramsey graph.
This discussion leads to the following conjecture as we think that the assumption $\chi$-unavoidable is not necessary here.
Some more evidence is provided in~\cite{RollinDiss}.
\begin{conjecture}\label{conj:OrderedFiniteForest}
 Let $(H,H')$ be a Ramsey finite pair of ordered forests. Then $\ors(H,H')$ contains a forest.
\end{conjecture}
With Theorem~\ref{thm:UnavoidInfinite} we show that the reverse statement of Conjecture~\ref{conj:OrderedFiniteForest} does not hold and that the family of all Ramsey finite pairs of ordered graphs might be rather diverse (see also Conjecture~\ref{conj:unavoidableConnected}).

Theorems~\ref{thm:UnavoidInfinite} and~\ref{thm:CaterpillarFinite} deal with pairs of connected $\chi$-unavoidable ordered forests.
The only pairs of connected $\chi$-unavoidable ordered graphs that we do not cover are formed by a right (left) star and an almost increasing right (left) caterpillar with defining sequence $d_2<d_1\leq d_3\leq\cdots\leq d_i$ for some $i\geq 3$.
We conjecture that these pairs are Ramsey finite.
A proof of the case $i=3$ and $\lvert E(H)\rvert=2$ of the following conjecture is given in~\cite{RollinDiss}.
\begin{conjecture}\label{conj:unavoidableConnected}
 Let $(H,H')$ be  a pair of $\chi$-unavoidable connected ordered graphs with at least two edges each.
 Then $(H,H')$ is Ramsey finite if and only if $(H,H')$ is a pair of a right star and an almost increasing right caterpillar or a pair of a left star and an almost increasing left caterpillar. 
\end{conjecture}
In Theorem~\ref{thm:CaterpillarFinite} we show that there are Ramsey finite pairs of ordered stars and ordered caterpillars of arbitrary diameter.
Again this is in contrast to the unordered setting where for any Ramsey finite pair $(H,H')$ of forests either one of $H$ or $H'$ is a matching or both are star forests (with additional constraints, see Theorem~\ref{thm:Faudree}).

\paragraph{Ramsey equivalence of ordered graphs.}
Finally we mention another line of research that might show an entire different behavior in the ordered setting than for unordered graphs.
Two ordered graphs are called \emph{Ramsey equivalent} if they have the same set of ordered Ramsey graphs.
This notion was introduced for graphs by Szab{\'o}~\textit{et al.}~\cite{Szabo_RamseyMinBipartite} and studied in several subsequent papers~\cite{OurRamseyEqu, BloomLiebenau, Fox_EquivClique}.
While it is easy to find Ramsey equivalent pairs of non-isomorphic graphs it remains open whether there is such a pair of connected graphs.
Surprisingly we do not know any Ramsey equivalent pair of non-isomorphic ordered graphs so far, even without the restriction on connectivity.
\begin{question}
 Are there non-isomorphic ordered graphs $H$ and $H'$ with $\ors(H)=\ors(H')$?
\end{question}
Figure~\ref{fig:K3OrderedNonEquivalent} shows that an ordered $K_3$ is not Ramsey equivalent to any ordered graph formed
by a union of $K_3$ and an isolated vertex.
\begin{figure}
 \centering
 \includegraphics{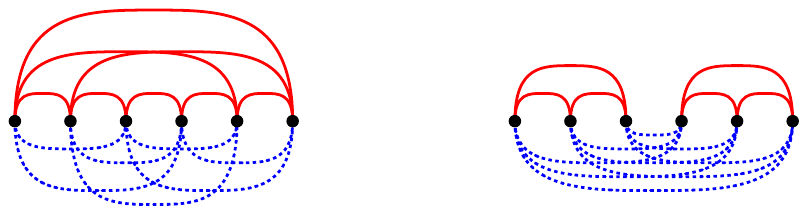}
 \caption[Colorings of an ordered $K_6$ without monochromatic copies of any ordered $K_3+ K_1$.]{Colorings of an ordered $K_6$ without monochromatic copies of $K_3\sqcup K_1$ (left) or a monochromatic copy of $K_3$ with an isolated vertex between its vertices  (right).}
 \label{fig:K3OrderedNonEquivalent}
\end{figure}
Nevertheless, we think that that $K_n$ is Ramsey equivalent to some ordering of a union of $K_n$ and an isolated vertex for sufficiently large $n$.
Further we observe here that for any ordered graph $H$ and each minimal ordered Ramsey graph $F$ of $H$ there are colorings $c_\ell$ and $c_r$ of the edges of $F$ such that each monochromatic copy of $H$ contains the leftmost vertex of $F$ under $c_\ell$ and the rightmost vertex of $F$ under $c_r$.
This shows that if $H$ and $H'$ are Ramsey equivalent and $H\subseteq H'$, then each copy of $H$ in $H'$ contains the leftmost and the rightmost vertex of $H'$.

\paragraph{Acknowledgements.}
We would like to thank Maria Axenovich and Yury Person for fruitful discussions that especially simplified Theorems~\ref{thm:ordRandomRamsey} and~\ref{thm:ordDensityRamseyInfinite}.

\bibliographystyle{abbrv}
\bibliography{lit}

\end{document}